\documentclass[11pt]{amsart}
\usepackage{graphicx,latexsym,amssymb,pinlabel,yhmath}
\usepackage[dvipsnames]{xcolor}
\usepackage[all]{xy}
\SelectTips{cm}{}

\usepackage{hyperref}
\hypersetup{colorlinks=true,citecolor=red,linkcolor=blue,urlcolor=black,filecolor=black}

\begin{document}
%
%
%
%
%
%

%
%

\theoremstyle{plain}
\newtheorem{theorem}{Theorem}[section]
\newtheorem{lemma}[theorem]{Lemma}
\newtheorem{proposition}[theorem]{Proposition}
\newtheorem{corollary}[theorem]{Corollary}
\newtheorem{definition}[theorem]{Definition}
\newtheorem{conjecture}[theorem]{Conjecture}
\newtheorem{problem}[theorem]{Problem}

\theoremstyle{definition}
\newtheorem{example}[theorem]{Example}
\newtheorem{remark}[theorem]{Remark}
\newtheorem{summary}[theorem]{Summary}
\newtheorem{notation}[theorem]{Notation}

\theoremstyle{remark}
\newtheorem{claim}[theorem]{Claim}
\newtheorem{sublemma}[theorem]{Sub-lemma}
\newtheorem{innerremark}[theorem]{Remark}

\numberwithin{equation}{section}
\numberwithin{figure}{section}

%
%

\renewcommand{\labelitemi}{--}

%
%

\def \N{\mathbb{N}}
\def \Z{\mathbb{Z}}
\def \Q{\mathbb{Q}}
\def \R{\mathbb{R}}
\def \C{\mathbb{C}}

\def \A{\mathcal{A}}
\def \As{\mathcal{A}^s}
\def \AY{\mathcal{A}^Y}
\def \AYc{\mathcal{A}^{Y,c}}
\def \btT{{}_{b}^t\mathcal{T}}
\def \Cob{\mathcal{C}ob}
\def \col{\hbox{\small color}}
\def \QCyl{\Q\mathcal{C}yl}
\def \F{\mathcal{F}}
\def \I{\mathcal{I}}
\def \LCob{\mathcal{LC}ob}
\def \dLCob{{}^d\!\mathcal{LC}ob}
\def \LqCob{\mathcal{LC}ob_q}
\def \dLqCob{{}^d\!\mathcal{LC}ob_q}
\def \M{\mathcal{M}}
\def \P{\mathcal{P}}
\def \qTCub{\mathcal{T}_q\mathcal{C}ub}
\def \QLCob{\Q\mathcal{LC}ob}
\def \QLqCob{\Q\mathcal{LC}ob_q}
\def \sLCob{{}^s\!\mathcal{LC}ob}
\def \sLqCob{{}^s\!\mathcal{LC}ob_q}
\def \T{\mathcal{T}}
\def \tsA{{}^{ts}\!\!\mathcal{A}}
\def \Ztilde{\widetilde{Z}}
\def \ZtildeY{\widetilde{Z}^{Y}}
\def \ZZ{\widetilde{\mathsf{Z}}}

\def \osqcup{\hphantom{}^<_\sqcup}

\def \aug{{\rm aug}}
\def \Aut{{\rm Aut}}
\def \cl{{\rm cl}}
\def \Coker{{\rm Coker}}
\def \deg{{\rm deg}}
\def \ideg{{\rm i\hbox{-}deg}}
\def \edeg{{\rm e\hbox{-}deg}}
\def \End{{\rm End}}
\def \Gr{{\rm Gr}}
\def \Hom{{\rm Hom}}
\def \incl{{\rm incl}}
\def \Id{{\rm Id}}
\def \Img{{\rm Im}}
\def \Ker{{\rm Ker}}
\def \Lie{{\rm Lie}}
\def \Lk{{\rm Lk}}
\def \mod{{\rm mod}}
\def \ord{{\rm ord}}
\def \pr{{\rm pr}}
\def \rk{{\rm rk}}
\def \sgn{{\rm sgn}}
\def \Sp{{\rm Sp}}
\def \Tors{{\rm Tors}}
\def \ud{{\rm d}}

\newcommand{\up}{\vspace{-0.5cm}}

\newcommand{\mediumdot}{{ \displaystyle \mathop{ \ \ }^{\hbox{$\centerdot$}}}}

\newcommand{\set}[1]{\lfloor #1\rceil}
\newcommand{\Star}[3]{\stackrel{{\scriptsize #1},{\scriptsize #2}}{\star}}

\newcommand{\figtotext}[3]{\begin{array}{c}\includegraphics[width=#1pt,height=#2pt]{#3}\end{array}}
\newcommand{\capleft}{\figtotext{10}{10}{capleft.eps}}
\newcommand{\cupright}{\figtotext{10}{10}{cupright.eps}}

\newcommand{\thetagraph}
{\hspace{-0.2cm} \figtotext{16}{16}{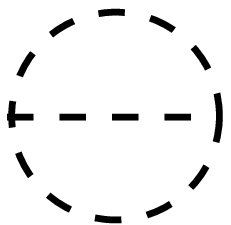} \hspace{-0.2cm}}

\newcommand{\strutgraphbot}[2]
{ \begin{array}{c} 
\phantom{.}\\[-5pt]
\labellist \small \hair 2pt 
\pinlabel {\scriptsize $#1$} [t] at 4 0
\pinlabel {\scriptsize $#2$} [t] at 180 0
\endlabellist
\includegraphics[scale=0.1]{one-chordbot}\\
\hphantom{ab}
\end{array}  }

\newcommand{\strutgraphtop}[2]
{ \begin{array}{c} 
\phantom{.}\\[-5pt]
\labellist \small \hair 2pt 
\pinlabel {\scriptsize $#1$} [t] at 4 140
\pinlabel {\scriptsize $#2$} [t] at 180 140
\endlabellist
\includegraphics[scale=0.1]{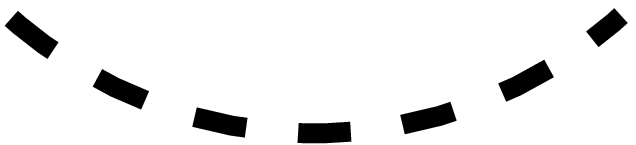}\\
\hphantom{ab}
\end{array}  }

\newcommand{\strutgraph}[2]
{ \! \! \begin{array}{c} 
\phantom{.}\\[-10pt]
\labellist \small \hair 2pt 
\pinlabel {\scriptsize $#1$} [l] at 37 18
\pinlabel {\scriptsize $#2$} [l] at 37 170
\endlabellist
\includegraphics[scale=0.1]{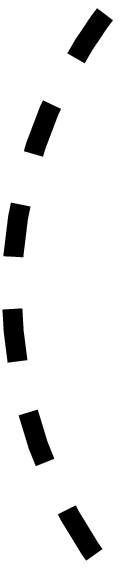} 
\end{array} \  \ }

\newcommand{\Ygraphtop}[3]
{  \begin{array}{c} 
\hphantom{.}\\
\labellist \small \hair 2pt 
\pinlabel {\scriptsize $#1$} [b] at 0 90
\pinlabel {\scriptsize $#2$} [b] at 90 90
\pinlabel {\scriptsize $#3$} [b] at 180 90
\endlabellist
\includegraphics[scale=0.18]{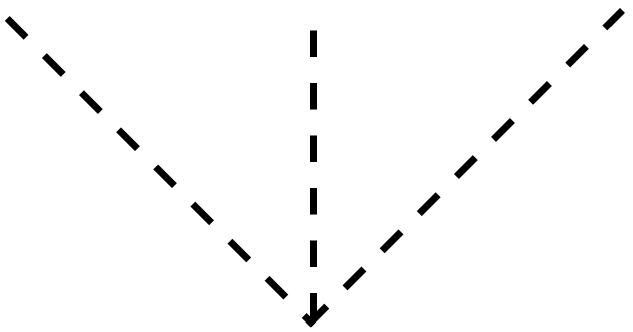}\\[-10pt]
\hphantom{.}
\end{array}  }

\newcommand{\bdraft}{\color{brown}{}}
\newcommand{\edraft}{\normalcolor{}}

\title[]{Splitting formulas for the LMO invariant\\ of rational homology three-spheres}

\date{April 16, 2014}

\author[]{Gw\'ena\"el Massuyeau}
\address{Institut de Recherche Math\'ematique Avanc\'ee, Universit\'e de Strasbourg \& CNRS,
7 rue Ren\'e Descartes, 67084 Strasbourg, France}
\email{massuyeau@math.unistra.fr}


\begin{abstract}
For rational homology $3$-spheres, there exist two universal finite-type invariants: 
the Le--Murakami--Ohtsuki invariant  and the Kontsevich--Kuperberg--Thurston invariant.
These  invariants take values in the same space of ``Jacobi diagrams'', but it is not known whether they are equal.
In 2004, Lescop proved that the KKT invariant satisfies some ``splitting formulas''
which relate the variations of  KKT  under replacement of embedded rational homology handlebodies by others in a ``Lagrangian-preserving'' way.
We show that the LMO invariant satisfies exactly the same relations.
The proof is based on the LMO functor, which is a generalization of the LMO invariant to the category of $3$-dimensional cobordisms,
and we generalize Lescop's splitting formulas to this setting.
\end{abstract}

\maketitle

\vspace{-0.5cm}

\section{Introduction}

A \emph{rational homology $3$-sphere} (or, \emph{$\Q$-homology $3$-sphere})
is a closed oriented $3$-manifold $S$  that has the same homology with rational coefficients 
as the standard $3$-sphere $S^3$.
Le, Murakami \& Ohtsuki defined in \cite{LMO} 
an invariant $Z(S)$ of rational homology $3$-spheres $S$ with values
in the algebra $\A(\varnothing)$ of Jacobi diagrams.
The LMO invariant $Z(S)$, which was originally denoted by $\hat{\Omega}(S)$ in \cite{LMO},
is multiplicative under connected sums. 
As shown in \cite{BGRT3}, it coincides with the Aarhus integral $\ring{A}(S)$  
introduced by Bar-Natan, Garoufalidis, Rozansky \& Thurston \cite{BGRT1,BGRT2}.
This paper is aimed at studying the behaviour of $Z$ under a certain type of 
rational homology handlebody replacement, called ``Lagrangian-preserving surgery'' 
by Lescop \cite{Lescop_KKT_split} and whose definition we now recall.

A \emph{rational homology handlebody} (or, \emph{$\Q$-homology handlebody}) of \emph{genus} $g$ 
is a compact oriented $3$-manifold $C'$ that has the same homology with rational coefficients as  the standard genus $g$ handlebody.
The \emph{Lagrangian} of $C'$ is the kernel $\mathbf{L}_{C'}^\Q$
of the homomorphism $\incl_*: H_1(\partial C';\Q) \to H_1(C';\Q)$ induced by the inclusion:
indeed, this is a Lagrangian subspace of $H_1(\partial C';\Q)$ with respect to the intersection pairing.
A \emph{$\Q$-Lagrangian-preserving pair} (or, \emph{$\Q$-LP pair})  is a pair $\mathsf{C}=(C',C'')$ of
two rational homology handlebodies whose boundaries are identified $\partial C' \equiv \partial C''$ 
in such a way that $\mathbf{L}_{C'}^\Q=\mathbf{L}_{C''}^\Q$. 
The \emph{total manifold} of the $\Q$-LP pair $\mathsf{C}$ is the closed oriented $3$-manifold
$$
C := (-C')\ {\cup}_{\partial}\ C''.
$$
Note that the inclusion $C' \subset C$ induces a canonical isomorphism $H_1(C';\Q)\simeq H_1(C;\Q)$.
The form $H^1(C;\Q)^{\otimes 3}\to \Q$ defined by triple-cup products 
$(x,y,z) \mapsto \langle x\cup y \cup z,[C]\rangle$ is skew-symmetric: we denote it by
$$
\mu\left(C\right) \in 
\Hom_\Q\left(\Lambda^3 H^1(C;\Q), \Q\right) \simeq \Lambda^3 H_1(C;\Q).
$$
Given a compact oriented $3$-manifold $M$ and a $\Q$-LP pair $\mathsf{C}=(C',C'')$ 
such that $C'$ is embedded in the interior of $M$,
one can replace the submanifold $C'$ in $M$ by $C''$ in order to obtain a new $3$-manifold
$$
M_{\mathsf{C}} := \left( M\setminus \operatorname{int} (C') \right) \cup_\partial C''.
$$
The move $M\leadsto M_{\mathsf{C}}$ between compact oriented $3$-manifolds is called a \emph{$\Q$-LP surgery}.

Suppose now that we are given a rational homology $3$-sphere $S$ and 
a finite family of  $\Q$-LP pairs $\mathsf{C}=(\mathsf{C}_1,\dots, \mathsf{C}_r)$ where
$C_i' \subset S$ and $C_i'\cap C_j' = \varnothing$ for all $i\neq j$.
We associate to the family $\mathsf{C}$ the  tensor
\begin{equation}
\label{eq:Y}
\mu(\mathsf{C}) := \mu(C_1) \otimes \cdots \otimes  \mu(C_r) 
\ \in \bigotimes_{i=1}^r \Lambda^3 H_1(C_i;\Q) \subset S^r \Lambda^3 H_1(C;\Q)
\end{equation}
where we have set $C:= C_1\sqcup \cdots \sqcup C_r$ so that $H_1(C;\Q)= H_1(C_1;\Q)\oplus \cdots \oplus H_1(C_r;\Q)$.
Besides, the linking number in $S$ defines for any $i\neq j$ a linear map
$$
\ell_{i,j}: H_1(C'_i;\Q) \times H_1(C'_j;\Q) \longrightarrow \Q
$$
by setting $\ell_{i,j}([K],[L]):= \Lk_S(K,L)$ for any oriented knots $K\subset C'_i$ and $L\subset C'_j$;
thus we can also associate to $\mathsf{C}$ the symmetric bilinear form
\begin{equation}
\label{eq:strut}
\ell_S(\mathsf{C}) := \sum_{i\neq j} \ell_{i,j}: H_1(C;\Q) \times H_1(C;\Q) \longrightarrow \Q,
\end{equation}
where $H_1(C;\Q)$ is identified to $H_1(C_1';\Q)\oplus \cdots \oplus H_1(C_r';\Q)$ in the canonical way.

We now recall how   symmetric products of antisymmetric $3$-tensors such as (\ref{eq:Y})
can be depicted graphically using \emph{Jacobi diagrams}.
For any $\Q$-vector space $V$, the space of  \emph{$V$-colored} Jacobi diagrams is
\begin{equation}\label{eq:Jacobi}
\A(V):= \frac{\Q\cdot\left\{\begin{array}{c}
\hbox{finite uni-trivalent graphs whose trivalent vertices are oriented}\\
\hbox{and whose univalent vertices are colored by $V$} 
\end{array}\right\}}{\hbox{AS, IHX, multilinearity}}.
\end{equation}
Here, an \emph{orientation} of a trivalent vertex is a cyclic ordering of the incident half-edges
(which, on pictures, is given by the counterclockwise direction)
and  the relations are\\[0.2cm]

\begin{center}
\labellist \small \hair 2pt
\pinlabel {AS} [t] at 102 -5
\pinlabel {IHX} [t] at 543 -5
\pinlabel {multilinearity} [t] at 1036 -5
\pinlabel {$= \ -$}  at 102 46
\pinlabel {$-$} at 484 46
\pinlabel {$+$} at 606 46
\pinlabel {$=0$} at 721 46 
\pinlabel {$+$} at 1106 46
\pinlabel {$=$} at 961 46
\pinlabel{$v_1+v_2$} [b] at 881 89
\pinlabel{$v_1$} [b] at 1042 89
\pinlabel{$v_2$} [b] at 1170 89
\pinlabel{.} at 1210 46
\endlabellist
\centering
\includegraphics[scale=0.35]{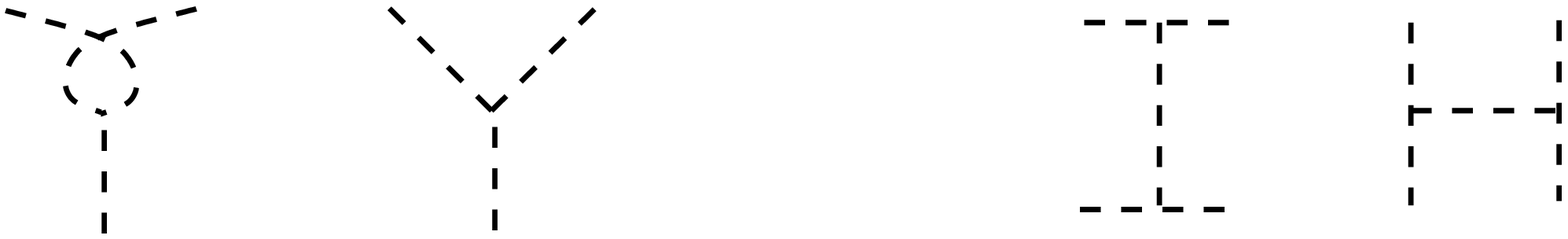}
\end{center}
\vspace{0.7cm}
With the disjoint union $\sqcup$ of diagrams, $\A(V)$ is a commutative algebra.
The \emph{internal degree} (or, \emph{i-degree}) of a Jacobi diagram is the number of trivalent vertices.
Then any symmetric product of antisymmetric $3$-tensors
$$(v_1 \wedge w_1 \wedge t_1) \cdots (v_r \wedge w_r \wedge t_r) \in S^r \Lambda^3 V$$
can be seen as the Jacobi diagram
\begin{equation}\label{eq:tripods}
\Ygraphtop{t_1}{w_1}{v_1} \sqcup \cdots \sqcup \Ygraphtop{t_r}{w_r}{v_r}
 \ \in \A(V).
\end{equation}
If we are given a symmetric bilinear form $\ell: V \times V \to \Q$,
then we can produce from  \eqref{eq:tripods} a new element of $\A(V)$ by matching pairwisely some of its univalents
vertices and  by multiplying the resulting diagram with the values of $\ell$ on the corresponding pairs of vertices. 
We shall say that we have \emph{glued with} $\ell$  some  legs  of  \eqref{eq:tripods}.
In particular, if $r$ is even, then we can glue with $\ell$ \emph{all}  legs of  \eqref{eq:tripods} to get an element of 
$$
\A(\varnothing) := \A(\{0\}) = 
\frac{\Q\cdot\left\{\begin{array}{c}
\hbox{finite trivalent graphs whose vertices are oriented}
\end{array}\right\}}{\hbox{AS, IHX}}.
$$ 
Note that $\A(\varnothing)$ is the algebra 
where the LMO invariant $Z$ of rational homology $3$-spheres takes values.
The above terminology being fixed, we can now state our main result.\\

\noindent
\textbf{Theorem.}
{\it Let $S$ be a rational homology $3$-sphere 
and let $\mathsf{C}=(\mathsf{C}_1,\dots, \mathsf{C}_r)$ be a finite family of $\Q$-LP pairs 
such that $C_i' \subset S$ and $C_i'\cap C_j' = \varnothing$ for all $i\neq j$.
For any $I\subset \{1,\dots,r\}$, 
we denote by $S_{\mathsf{C}_I}$  the manifold obtained from $S$ by 
the $\Q$-LP surgeries $S \leadsto S_{\mathsf{C}_i}$ performed simultaneously for all $i\in I$.
Then we have the following ``{splitting formula}'':
\begin{equation}
\label{eq:alternate_sum}
\sum_{I \subset \{1,\dots,r\} } (-1)^{|I|} \cdot Z\left(S_{\mathsf{C}_I}\right)
= \Bigg(\!\!\begin{array}{c} \hbox{\small sum of all ways of gluing  }\\
\hbox{\small all legs of $\mu(\mathsf{C})$ with $\ell_S(\mathsf{C})/2$ } \end{array}\!\!\Bigg) + (\ideg >r),
\end{equation}
being understood that the above sum is zero when $r$ is odd.}\\

This theorem generalizes the fact that  the LMO invariant of rational homology $3$-spheres
 is universal among $\Q$-valued finite-type invariants
in the sense of Ohtsuki and Goussarov--Habiro \cite{Le_universality,Habiro}.
Indeed, 
finite-type invariants in this sense can be formulated in terms of $\Z$-LP surgeries \cite{AL}.
(A \emph{$\Z$-LP surgery} is defined in a way similar to a $\Q$-LP surgery
except that rational homology is replaced by  integral homology.)
However, the notion of ``finite-type invariant'' differs  if one formulates it 
in terms of $\Q$-LP surgeries instead of $\Z$-LP surgeries: 
this difference has been recently analyzed by Moussard in the case of rational  homology $3$-spheres  \cite{Moussard}.
Let us observe that, in contrast with $\mathbb{Z}$-LP surgery, 
$\Q$-LP surgery relates any two rational  homology $3$-spheres:
thus  $\Q$-LP surgery is more appropriate
if one wants to consider rational  homology $3$-spheres \emph{all together.}

The analogue of the above theorem for the Kontsevich--Kuperberg--Thurston invariant  
$Z^{\operatorname{KKT}}$ 
has already been proved by Lescop: see \cite{Lescop_KKT_split} and \cite[\S 3]{Lescop_KKT_surgery}.
However, it is not known whether $Z^{\operatorname{KKT}}=Z$ in general.
Lescop's ``splitting formula'' for $Z^{\operatorname{KKT}}$ generalizes her ``sum formula'' 
for the Casson--Walker invariant  \cite{Lescop_CW}.
Indeed, according to \cite{Lescop_KKT_split} and \cite{LMO}, we have
$$
\hbox{i-degree $2$ part of }  Z^{\operatorname{KKT}}(S) 
= \frac{\lambda_{\operatorname{W}}(S)}{4} \cdot \thetagraph 
= \hbox{i-degree $2$ part of }  Z(S)
$$
where $\lambda_{\operatorname{W}}(S)$ denotes  Walker's extension of  the Casson invariant  as normalized in \cite{Walker}.

We shall prove the above theorem using the LMO functor $\Ztilde$:
this is a functorial extension of the LMO invariant to $3$-manifolds with boundary, 
which has been introduced in a previous joint work with Cheptea \& Habiro \cite{CHM}.
The possibility of such a proof has been announced in \cite[Remark 7.12]{CHM}.
The main features of the LMO functor are recalled in \S \ref{sec:LMO_functor}
but, in a few words, let us recall that it is defined on the category of  so-called ``Lagrangian cobordisms'' 
which are homology handlebodies with appropriate parameterizations of their boundaries,  and it takes values in a certain category of Jacobi diagrams. 
We state in \S \ref{sec:general} a generalized version of the  above theorem,
where the $\Q$-homology $3$-sphere $S$ is replaced by any  $\Q$-Lagrangian cobordism $M$ 
and the LMO invariant $Z$ is replaced by the LMO functor $\Ztilde$.
This results in ``generalized splitting formulas'' involving a notion $\Lk_M^{\mathbf{E}}(-,-)$ of ``linking number'' in $M$,
which depends on the choice of an isotropic subspace $\mathbf{E} \subset H_1(\partial M;\Q)$  such that $H_1(\partial M;\Q) = \mathbf{L}_M^\Q \oplus \mathbf{E}$.
Note that the category of Lagrangian cobordisms includes the monoid of homology cylinders,
so that our results apply  in particular to the LMO homomorphism studied in \cite{HabiroMassuyeau,MM} with a natural notion of linking number.
The generalized splitting formulas are proved in \S \ref{sec:proof} using the properties of the LMO functor established in \cite{CHM}.
The proof also needs several intermediate results, which can be of independent interest and are included in two appendices.
Appendix \ref{sec:lk} gives some properties of the generalized linking number  $\Lk_M^{\mathbf{E}}(-,-)$ 
and it inspects the dependence on $\mathbf{E}$.
Appendix \ref{sec:Milnor} shows that the Milnor's triple linking numbers of an algebraically-split link in a $\Q$-homology $3$-sphere 
are encoded in the ``Y'' part of the Kontsevich--LMO invariant.
The latter result is in the continuity of the work of Habegger \& Masbaum \cite{HabeggerMasbaum,Moffatt}.\\

\noindent
{\it Acknowledgements.} This work was partially supported by the French ANR research project ANR-08-JCJC-0114-01.
We would like to thank Christine Lescop for her comments and for pointing out a small gap in a previous version of Appendix \ref{sec:Milnor}.\\

\noindent
{\it Conventions.} 
The boundary $\partial N$ of an oriented manifold $N$ is always oriented using the ``outward normal  first'' rule.

The boundary of a compact oriented $3$-manifold $M$ is said to be \emph{parameterized} 
by a closed oriented surface $F$ if $M$ comes with a continuous map $m:F \to M$ that is an orientation-preserving homeomorphism onto $\partial M$;
the lower-case letter $m$ will also denote the corresponding homeomorphism  $F \to \partial M$;
 we sometimes omit the boundary parameterization in our notation and denote the pair $(M,m)$ simply by $M$.
 
Implicitly, compact oriented $3$-manifolds with parameterized boundary are considered up to homeomorphisms that preserve orientations and boundary parameterizations.
 Similarly, tangles in $3$-manifolds are considered up to isotopy.

\section{Review of the LMO functor} \label{sec:LMO_functor}

In this section, we briefly sketch the construction of the LMO functor.
The reader is referred for further details to the paper \cite{CHM}.
Here, our exposition is only intended to sum up the various steps of the construction 
using the same notations as in \cite{CHM}.

\subsection{The category of $\Q$-Lagrangian cobordisms}    \label{subsec:Q-LCob}

We start by describing the source of the LMO functor and, for this, 
we  consider the category $\Cob$ of $3$-dimensional cobordisms introduced by Crane \& Yetter \cite{CY,Kerler}.
By definition, an object of $\Cob$ is an integer $g\geq 0$,
which one thinks as the genus of a  compact connected oriented surface  with one boundary component.
We actually \emph{fix} a model $F_g$ for such a surface and we identify $\partial F_g$ with the square $S:= \partial([-1,1]^2)$.
For any integers $g_+\geq 0$ and $g_-\geq 0$, a morphism $g_+\to g_-$ in the category $\Cob$ 
is  a cobordism $(M,m)$ from the surface $F_{g_+}$ to the surface $F_{g_-}$:
more precisely, $M$ is a compact connected oriented $3$-manifold together with a {boundary parameterization}  
$$
m: -F_{g-} \cup_{S \times \{-1\}} \big( S \times [-1,1]\big) \cup_{S \times \{1\}} F_{g_+} \longrightarrow M.
$$
Thus, the boundary parameterization $m$ restricts to two embeddings $m_-: F_{g_-}\to M$ and $m_+: F_{g_+} \to M$,
whose images are called \emph{bottom surface} and \emph{top surface} respectively.
The composition $\circ$ in $\Cob$ is given by ``vertical'' gluing of cobordisms,
while the ``horizontal'' gluing of cobordisms defines a strict monoidal structure $\otimes$ on that category.
(Note that, to define the latter operation, we assume that the model surfaces $F_0,F_1,F_2,\dots$  
come with an identification of $F_{h+h'}$ with  $F_h\, \sharp_\partial\, F_{h'}$ for any $h,h'\geq 0$,
where the boundary connected sum $\sharp_\partial$ is  performed along  the segments $\{1\} \times [-1,1] \subset \partial F_h$ and  $\{-1\} \times [-1,1] \subset \partial F_{h'}$.)

\begin{figure}[h]
\begin{center}
\labellist \small \hair 0pt 
\pinlabel {$\alpha_1$} [l] at 310 105
\pinlabel {$\alpha_g$} [l] at 568 65
\pinlabel {$\beta_1$} [b] at 216 80
\pinlabel {$\beta_g$} [b] at 472 69
\pinlabel {$\circlearrowleft$} at 520 30
\pinlabel {$x$} [l] at 715 6
\pinlabel {$y$} [bl] at 695 43
\pinlabel {$z$} [b] at 656 66
\endlabellist
\includegraphics[scale=0.5]{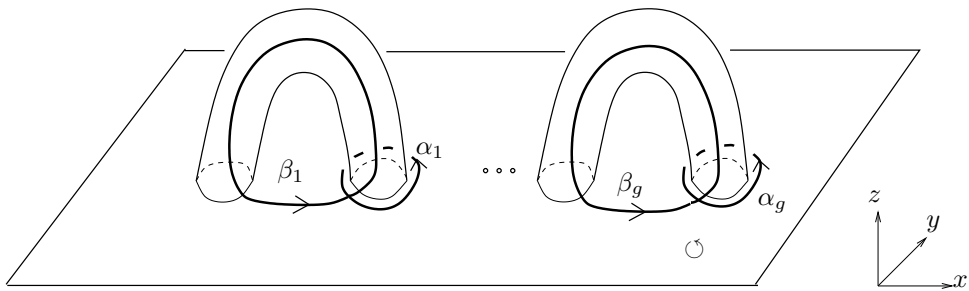}
\end{center}
\caption{The model surface  $F_g$ of genus $g\geq 0$, 
 with its system of meridians and parallels $(\alpha,\beta)$.}
\label{fig:surface}
\end{figure}

The study of these cobordisms can be reduced  to the study of some kind of tangles.
To do this, we need to \emph{choose}  a system  of ``meridians'' $(\alpha_1,\dots, \alpha_g)$ 
and ``parallels'' $(\beta_1,\dots,\beta_g)$ on each model surface $F_g$ (in a way compatible with the identifications $F_{h+h'} \equiv F_h\, \sharp_\partial\, F_{h'}$ for all $h,h'\geq 0$).
In order to fix ideas, we now assume that the surface $F_g$ is embedded in the ambient space $\R^3 \subset S^3$ (with cartesian coordinates $x,y,z$):
specifically, $F_g$ is obtained from the ``horizontal'' square $[-1,1] \times [-1,1] \times \{0\}$ by ``adding''  along the $x$ coordinate $g$ handles contained in the half-space $z>0$:
see Figure \ref{fig:surface} where the system of meridians and parallels $(\alpha,\beta)$ is also shown. 
Consider the submanifold   $C_{g_-}^{g_+}\subset \R^3$ 
obtained from  the standard cube $[-1,1]^3$ by ``digging $g_-$ tunnels at the bottom side'' and ``attaching $g_+$ handles on the top side'',  as shown on Figure \ref{fig:cube}. 
The oriented boundary of $C_{g_-}^{g_+}$ consists of one copy of $F_{g_+}$ and one copy of $-F_{g_-}$ 
(which are obtained from the model surfaces by translation in the $z$ direction) joined by the cylinder $S \times [-1,1]$.
Thus, for any  $(M,m) \in \Cob(g_+,g_-)$, the source of the boundary parameterization $m$ is  $\partial C_{g_-}^{g_+}$.
The boundary of $M$ can be ``killed'' by attaching $(g_-+g_+)$ handles of index $2$ along $\partial M$:
more precisely, we attach one $2$-handle along each curve $m_-(\alpha_i)$ of the bottom surface
and one $2$-handle along each curve $m_+(\beta_j)$ of the top surface.
What we get is a compact oriented $3$-manifold $B$ with boundary parameterization $b:\partial C_0^0  \to B$,
i.e$.$ a morphism $(B,b) \in \Cob(0,0)$.
Furthermore, the cocores of the attached $2$-handles define a $(g_++g_-)$-component 
framed oriented tangle $\gamma$ in $B$ 
with components $\gamma_1^+,\dots,\gamma_{g_+}^+$ ``on the top'' 
and components $\gamma_1^-,\dots,\gamma_{g_-}^-$ ``on the bottom'':
the pair $(B,\gamma)$ is called in \cite{CHM} a \emph{bottom-top tangle}.
For example, if we do the previous operation on $C_{g_-}^{g_+}\in \Cob(g_+,g_-)$,
then we obtain the trivial  bottom-top tangle in the standard cube $C_{0}^0=[-1,1]^3$. 
There is a category $\btT$ of bottom-top tangles whose composition law is defined 
in such a way that the previous construction $(M,m)\mapsto ((B,b),\gamma)$ defines an isomorphism of categories
\begin{equation}\label{eq:btT_to_Cob}
\btT \stackrel{\simeq}{\longrightarrow} \Cob.
\end{equation}

\begin{figure}[h]
\begin{center}
\labellist \small \hair 2pt 
\pinlabel {$1$} at 177 45
\pinlabel {$1$} at 176 285
\pinlabel {$g_-$} at 383 44
\pinlabel {$g_+$} at 383 284
\pinlabel {$F_{g_+}$} [l] at 560 408
\pinlabel {$F_{g_-}$} [l] at 590 0
\pinlabel {$x$} [l] at 654 95
\pinlabel {$y$} [bl] at 633 130
\pinlabel {$z$} [b] at 596 153
\endlabellist
\includegraphics[scale=0.38]{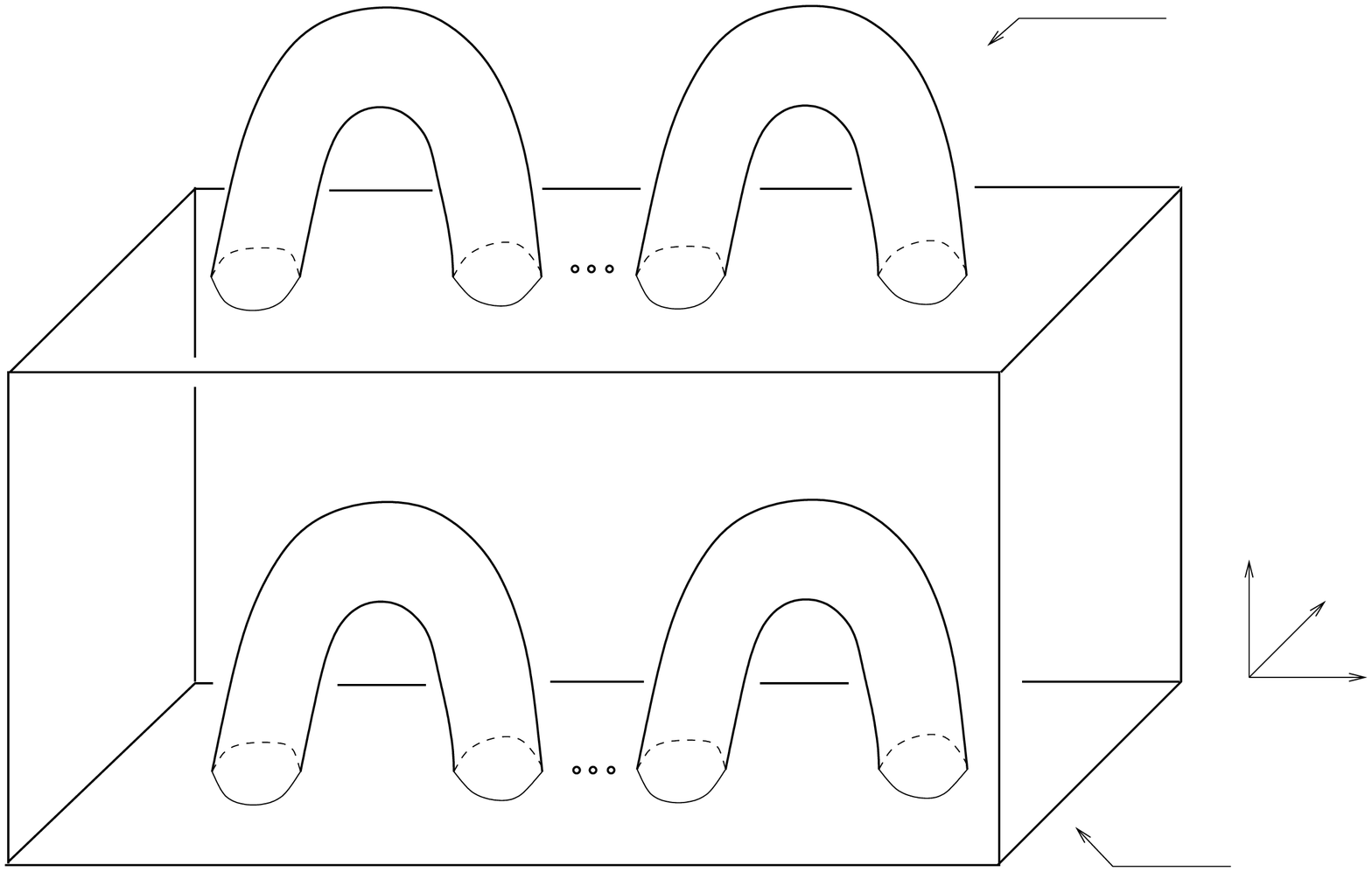}
\end{center}
\caption{The cube $C_{g_-}^{g_+}$ with $g_-$ tunnels and $g_+$ handles.}
\label{fig:cube}
\end{figure}

Unfortunately, the LMO functor is not defined on the full category $\Cob$,
but only on the subcategory $\QLCob$ of \emph{$\Q$-Lagrangian cobordisms}.
The definition of this subcategory needs to \emph{fix} a Lagrangian subspace $A_g^\Q$ of $H_1(F_g;\Q)$ for any integer $g\geq 0$.
A cobordism $(M,m) \in \Cob(g_+,g_-)$ belongs to $\QLCob(g_+,g_-)$  if and only if 
\begin{enumerate}
\item[(1)] $H_1(M;\Q) = m_{-,*}(A_{g_-}^\Q) + m_{+,*}(H_1(F_{g_+};\Q))$,
\item[(2)] $m_{+,*}(A_{g_+}^\Q) \subset m_{-,*}(A_{g_-}^\Q)$ as subspaces of $H_1(M;\Q)$.
\end{enumerate}
Concretely, we shall take  $A_g^\Q$ to be the subspace of $H_1(F_g;\Q)$ spanned by  $(\alpha_1,\dots,\alpha_g)$,
and we also consider the subspace $B_g^\Q$ spanned by  $(\beta_1,\dots,\beta_g)$. Then, in presence of (2),
condition (1) can be replaced by the following:
\begin{enumerate}
\item[(1')] $m_{+,*} \oplus m_{-,*}:   B_{g_+}^\Q \oplus A_{g_-}^\Q \to H_1(M;\Q)$ is an isomorphism.
\end{enumerate}

\begin{example} \label{ex:IC}
Let $g\geq 0$ be an integer.
A  \emph{$\Q$-homology cylinder} over the surface $F_g$ is a cobordism $(M,m) \in \Cob(g,g)$ such that $m_{\pm,*}: H_1(F_g;\Q) \to H_1(M;\Q)$ is an isomorphism and $m_{+,*}=m_{-,*}$.
The set of $\Q$-homology cylinders constitutes a submonoid $\QCyl(F_g)$ of the monoid $\QLCob(g,g)$.
\end{example} 

Assume that $((B,b),\gamma)$ is the  tangle corresponding to a cobordism $(M,m) \in \Cob(g_+,g_-)$ 
via the description \eqref{eq:btT_to_Cob}:
then $(M,m)$ belongs to $\QLCob(g_+,g_-)$ if and only if 
$B$ is a $\Q$-homology cube (i.e$.$ it has the same $\Q$-homology as the standard cube $[-1,1]^3=C_0^0$)
and the linking matrix $\Lk_B(\gamma^+)$ is trivial.
Here the \emph{linking matrix}  $\Lk_B(\gamma)$ of the framed oriented tangle $\gamma$ in $B$  is defined by
$$
\left(\begin{array}{c|c}
 \Lk_B(\gamma^+) & \Lk_B(\gamma^+,\gamma^-) \\ \hline
\Lk_B(\gamma^-,\gamma^+)  & \Lk_B(\gamma^-) 
\end{array}\right)
= \Lk_B(\gamma) := \Lk_{\hat B} (\hat \gamma) =
\left(\begin{array}{c|c}
 \Lk_{\hat B}(\hat \gamma^+) & \Lk_{\hat B}(\hat \gamma^+,\hat \gamma^-) \\ \hline
\Lk_{\hat B}(\hat \gamma^-,\hat \gamma^+)  & \Lk_{\hat B}(\hat \gamma^-) 
\end{array}\right)
$$
where  $\hat \gamma$ is the ``plat'' closure  of $\gamma$ in the $\Q$-homology $3$-sphere 
\begin{equation}\label{eq:B_hat}
\hat B:= B \cup_b \left(S^3 \setminus \operatorname{int}(C_0^0)\right),
\end{equation}
and the rows/columns of $\Lk_B(\gamma)$ are indexed   by the set $\pi_0(\gamma)=\pi_0(\gamma^+)\cup \pi_0(\gamma^-)$ of connected components of $\gamma$.

To be fully exact,
the source of the LMO functor is the category $\QLqCob$ of \emph{$\Q$-Lagrangian $q$-cobordisms}.
An object of $\QLqCob$ is a non-associative word in the single letter $\bullet$.
For any two such words $w_+$ and $w_-$, a morphism $w_+ \to w_-$ in the category $\QLqCob$ is 
a $\Q$-Lagrangian cobordism from $F_{g_+}$ to $F_{g_-}$ where $g_\pm$ is the length of $w_\pm$.
The category $\QLqCob$ is a monoidal category in the non-strict sense.

\subsection{The category of top-substantial Jacobi diagrams}

We now describe the target of the LMO functor and, for this, we need to fix some terminology.
For any finite set $C$, we denote by $\A(C)$ the space of Jacobi diagrams colored by $C$.
With the notation \eqref{eq:Jacobi} of the Introduction,  we have $\A(C):= \A(\Q \!\cdot\! C)$ where 
$\Q\!\cdot\! C$ is the vector space spanned by  $C$. 
We shall also  need the degree completion of $\A( C)$ which we  denote in the same way.
Here the \emph{degree} of a Jacobi diagram is half the total number of its vertices;
 a Jacobi diagram of degree $1$
$$
\strutgraph{c_1}{c_2}
\quad (\hbox{where} \ c_1,c_2 \in C)
$$
is called a \emph{strut}.
Any rational matrix $M=(m_{ij})_{i,j\in C}$, whose rows/columns are indexed by $C$,
defines a linear combination of struts by setting
$$
M := \sum_{i,j \in C} m_{ij}\, \strutgraph{i}{j}.
$$ 
If $S$ is another finite set, a Jacobi diagram in $\A(C \cup S)$ is \emph{$S$-substantial}
if it does not contain any strut whose two ends are colored by $S$.

The category of \emph{top-substantial Jacobi diagrams} is the linear category $\tsA$
whose objects are integers $g\geq 0$ and whose space of morphisms $\tsA(g,f)$ is,
for any integers $g\geq 0$ and $f\geq 0$, the subspace of $\A(\set{g}^+\cup \set{f}^-)$
spanned by $\set{g}^+$-substantial Jacobi diagrams.
Here $\set{g}^+$ denotes the $g$-element finite set $\{1^+,\dots,g^+\}$
while  $\set{f}^-$ denotes the $f$-element finite set $\{1^-,\dots,f^-\}$.
For any integers $f,g,h\geq 0$, the composition law $\circ$ of $\tsA$
is defined for any Jacobi diagrams $D\in \tsA(g,f) $ and $E \in \tsA(h,g)$ by
$$
D \circ E := 
\left(\begin{array}{cc}
\hbox{sum of all ways of gluing \emph{all} the $i^+$-colored vertices of $D$}\\
\hbox{to \emph{all} the $i^-$-colored vertices of $E$, for every $i=1,\dots,g$}
\end{array}\right) \ \in \tsA(h,f) .
$$
There is also a tensor product $\otimes$ in the category $\tsA$ defined by 
the disjoint union of diagrams $\sqcup$ and the appropriate shifts of colors. 
Thus the category $\tsA$  is monoidal  in the strict sense.

\subsection{Sketch of the construction}

The LMO functor of  \cite{CHM} is a tensor-preserving functor
$$
\Ztilde: \QLqCob \longrightarrow \tsA.
$$
Its construction uses the description \eqref{eq:btT_to_Cob} of cobordisms
and it can be sketched as follows.

Let $(M,m)\in \QLqCob(w,v)$ where $w$ and $v$ are non-associative words in the single letter $\bullet$
of length $g$ and $f$ respectively.
Let also $(B,\gamma):=((B,b),\gamma)$ be the bottom-top tangle 
in a $\Q$-homology cube corresponding to $M$ via \eqref{eq:btT_to_Cob}.
The framed tangle $\gamma$ can be promoted to a ``$q$-tangle'' in the sense of \cite{LM1,LM2} 
by transforming $w$ and $v$ into non-associative words in the letters $+,-$ by the rule $\bullet \mapsto (+-)$.
Then the Kontsevich--LMO invariant of $(B,\gamma)$
$$
Z(B,\gamma) \in \A(\gamma)
$$
is defined in the (degree completion of the) space $\A(\gamma)$ of Jacobi diagrams 
\emph{based} on the oriented $1$-manifold underlying $\gamma$.
The Kontsevich--LMO invariant $Z(B,\gamma)$ of $q$-tangles in $\Q$-homology cubes,
which we are using here, has the following features:
\begin{itemize}
\item If $B$ is the standard cube $[-1,1]^3$, then $Z([-1,1]^3,\gamma)$ 
coincides with the usual Kontsevich integral $Z(\gamma)$ of $q$-tangles,
as it is normalized in \cite[\S 3.4]{CHM}. (Note that this normalization  differs from \cite{LM1,LM2}.)
\item If $\gamma$ is empty,  then $Z(B,\varnothing)$ coincides 
with the usual LMO invariant $Z(\hat B)$ of the $\Q$-homology $3$-sphere $\hat B$.
(Note that this invariant is denoted by $\hat\Omega(\hat B)$ in \cite{LMO} and by $\mathring{A}(\hat B)$ in \cite{BGRT1,BGRT2}.)
\end{itemize}
The Kontsevich--LMO  invariant $Z(B,\gamma)$ is constructed from the usual Kontsevich integral 
$Z(L \cup \gamma) \in \A(L\cup \gamma)$,
where $(L,\gamma)$ is a \emph{surgery presentation} of $(B,\gamma)$:
thus, $L \subset [-1,1]^3$ is a framed link and $\gamma \subset [-1,1]^3$ is a framed oriented tangle disjoint from $L$
such that surgery along $L$ transforms $([-1,1]^3,\gamma)$ into $(B,\gamma)$.
The passage $Z(L \cup \gamma) \leadsto Z(B,\gamma)$ can be performed in two ways: 
one can either follow the original LMO approach \cite{LMO},
or use the formal Gaussian integration of \cite{BGRT1,BGRT2}. The latter approach is adopted in \cite[\S 3.5]{CHM}.
Next, by applying the diagrammatic Poincar\'e--Birkhoff--Witt isomorphism $\chi:\A(\pi_0(\gamma)) \to \A(\gamma)$, we get
$$
\chi^{-1}Z(B,\gamma) \in \A(\pi_0(\gamma))=\A(\set{g}^+ \cup \set{f}^-),
$$
where the finite set $\pi_0(\gamma)= \pi_0(\gamma^+) \cup \pi_0(\gamma^-)$
is identified with $\set{g}^+ \cup \set{f}^-$ in the obvious way. 
The fact that $\Lk_B(\gamma^+)=0$ implies that $\chi^{-1}Z(B,\gamma)$
is actually an element of $\tsA(g,f)$. 
Finally, the LMO functor is defined on the $\Q$-Lagrangian $q$-cobordism $M$ by 
$$
\Ztilde(M) := \chi^{-1}Z(B,\gamma) \circ \mathsf{T}_g \ \in \tsA(g,f)
$$
where $\mathsf{T}_g\in \tsA(g,g)$ is a constant  that is defined in an appropriate way 
from the Baker--Campbell--Hausdorff series.

Finally, let us recall that the series of Jacobi diagrams $\Ztilde(M) \in \tsA(g,f)$ can be decomposed as follows.
Let $\A^Y(\set{g}^+ \cup \set{f}^-)$ be the subspace of $\A(\set{g}^+ \cup \set{f}^-)$ 
spanned by Jacobi diagrams without strut component,
and let $\A^s(\set{g}^+ \cup \set{f}^-)$ be the subspace  spanned by disjoint unions of struts.
Then, we have
\begin{equation}\label{eq:sqcup}
\  \Ztilde(M) = \Ztilde^s(M)\, \sqcup\, \Ztilde^Y(M) \ \hbox{where} \ 
\left\{\begin{array}{l}
\! \Ztilde^s(M) = \exp_\sqcup\big(\frac{\Lk_B(\gamma)}{2}\big) \in \A^s(\set{g}^+ \cup \set{f}^-),\\
\! \Ztilde^Y(M) \in \A^Y(\set{g}^+ \cup \set{f}^-).
\end{array}\right.
\end{equation}

\section{Generalized splitting formulas for the LMO functor}   \label{sec:general}

In this section, we state a generalized version of Lescop's splitting formulas.
These ``generalized splitting formulas'' apply  to the LMO functor of  $\Q$-Lagrangian cobordisms.

\subsection{Linking numbers in $\Q$-Lagrangian cobordisms}

The first lemma implies that, if one forgets about their boundary parameterizations, 
$\Q$-Lagrangian cobordisms are just $\Q$-homology handlebodies. 

\begin{lemma}\label{lem:QLCob_to_QHH}
Let $(M,m) \in \QLCob(g,f)$ 
and let $N$ be a compact oriented $3$-manifold which is embedded in $M$.
If $N$ and $\partial N$ are connected,
then  $N$ is a $\Q$-homology handlebody. 
\end{lemma}

\begin{proof}
By doing the composition $C^{f}_0 \circ M \circ C_{g}^0$ in the category $\QLCob$,
we can assume that $g=f=0$, i.e$.$ $M$ is a $\Q$-homology cube. 
The map $\incl_*: H_1(\partial N;\Q) \to H_1(N;\Q)$ is surjective:
for any $x=[X] \in H_1(N;\Q)$, there exists a rational $2$-chain $Y$  in $M$ such that $\partial Y =X$
and $Y$ is transversal to $\partial N$; then, $X$ is homologous to $Y \cap \partial N$ in $N$.
Therefore $H^2(N;\Q) \simeq H_1(N, \partial N;\Q) =0$ which implies that $H_2(N;\Q)=0$.
Let $r$ be the genus of the surface $\partial N$;
since $ \chi(N) = \chi(\partial N)/2 = 1-r$, we deduce that $\dim H_1(N;\Q)=r$.
Thus $N$ is a $\Q$-homology handlebody of genus $r$.
\end{proof}

Therefore we can apply the results of  Appendix \ref{sec:lk} to any cobordism $M \in \QLCob(g,f)$.
Thus, there is a notion of linking number $\Lk_M^{m_*(\mathbf{E})}(-,-)$  in $M$ for every choice of an $M$-essential subspace $\mathbf{E}$ of
$$
H_1(\partial C_{f}^{g};\Q)\simeq \underbrace{H_1(F_{g};\Q)}_{\hbox{\scriptsize ``top''}}  \oplus \underbrace{H_1(F_{f};\Q)}_{\hbox{\scriptsize ``bottom''}}.
$$
In particular, we are interested in the  $M$-essential subspace 
$$\mathbf{B\!A} := \mathbf{B\!A}(g,f) = B_{g}^\Q \oplus A_{f}^\Q \ \subset H_1(\partial C_{f}^{g};\Q) $$ 
which is also isotropic. 
The resulting link invariant  $\Lk_M(-,-):=\Lk_M^{m_*(\mathbf{B\!A})}(-,-)$  can be reduced as follows
to the usual notion of linking number in a $\Q$-homology $3$-sphere.

\begin{lemma}\label{lem:lk}
Let $M \in \QLCob(g,f)$ and let $(B,\gamma) \in \btT(g,f)$ be the corresponding bottom-top tangle. Then,
for any disjoint oriented knots $K,L \subset M \subset B$, we have
$$
\Lk_M(K,L) = \Lk_{\hat B}(K,L)
$$
where $\hat B$ is the $\Q$-homology $3$-sphere  defined by \eqref{eq:B_hat}.
\end{lemma}

\begin{proof}
Let $D$ be a rational $2$-chain in $M$ such that $\partial D =  L - \tilde L$, 
where $ \tilde L$ is a rational $1$-cycle in $\partial M$ such that $[ \tilde L]\in H_1(\partial M;\Q)$  belongs to $m_*(\mathbf{B\!A})$.
Since each of the curves $m_+(\beta_1),\dots, m_+(\beta_{g})$ and $m_-(\alpha_1),\dots, m_-(\alpha_{f})$
bounds a disk in $B \subset \hat B$, the $2$-chain $D$ in $M$ can be ``filled'' to a $2$-chain $D'$ in $\hat B$
such that $\partial D'= L$ and $K \mediumdot_{\!\hat B}\, D'= K \mediumdot_{\!M}\, D$. Therefore we obtain
$$
\Lk_M(K,L) \stackrel{\eqref{eq:lk_as_intersection}}{=}  K \mediumdot_{\!M}\, D =  K \mediumdot_{\! \hat B}\, D' = \Lk_{\hat B}(K,L).
$$ 

\up
\end{proof}

Although the subspace $\mathbf{B\!A}(g,f)$ has the advantage to be $M$-essential for any cobordism $M \in \QLCob(g,f)$,
it is sometimes more natural to use other essential subspaces.

\begin{example} \label{ex:D}
Consider the submonoid $\QCyl(F_g)$ of $\QLCob(g,g)$ introduced in Example~\ref{ex:IC} and set $H_\Q:= H_1(F_g;\Q)$.
The following subspaces of $H_1(\partial C_{g}^{g};\Q) \simeq H_\Q  \oplus  H_\Q$ are $M$-essential for any   $M \in \QCyl(F_g)$:
$$
\mathbf{D} := \big\{(h,h)\, \vert\, h \in H_\Q  \big\},  \quad  \mathbf{E}_+ := \big\{(h,0)\, \vert\, h \in H_\Q  \big\}, \quad  \mathbf{E}_- := \big\{(0,h)\, \vert\, h \in H_\Q  \big\}
$$ 
The corresponding notions of linking number  $\Lk_M^{m_*(\mathbf{D})}(-,-)$  and $\Lk_M^{m_*(\mathbf{E}_\pm)}(-,-)$ 
coincide with the  invariants denoted  in  \cite[Appendix B]{MM} by $\Lk(-,-)$ and $\Lk_\mp(-,-)$, respectively. 
We are particularly interested in the subspace $\mathbf{D}$ which is isotropic (in contrast with $\mathbf{E}_\pm$).
When  $M= F_g \times [-1,1]$, the number  $\Lk(K,L):=\Lk_M^{m_*(\mathbf{D})}(K,L)$ 
can be computed locally by considering knot diagrams  on the bottom surface $F_g \times \{-1\}$:\\[-0.1cm]
\begin{equation}\label{eq:Lk_cylinder}
\Lk(K,L) := 
\frac{1}{2} \left(\ \sharp \  \begin{array}{c}
\labellist
\small\hair 2pt
 \pinlabel {$K$} [br] at 1 19
 \pinlabel {$L$} [bl] at 19 20
\endlabellist
\includegraphics[scale=1.0]{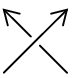}
\end{array}
+ \sharp \  \begin{array}{c}
\labellist
\small\hair 2pt
 \pinlabel {$L$} [br] at 1 19
 \pinlabel {$K$} [bl] at 19 20
\endlabellist
\includegraphics[scale=1.0]{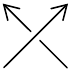}
\end{array}\ \right)
- \frac{1}{2}\left(\ \sharp \  \begin{array}{c}
\labellist
\small\hair 2pt
 \pinlabel {$K$} [br] at 1 19
 \pinlabel {$L$} [bl] at 19 20
\endlabellist
\includegraphics[scale=1.0]{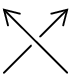}
\end{array}
+ \sharp \  \begin{array}{c}
\labellist
\small\hair 2pt
 \pinlabel {$L$} [br] at 1 19
 \pinlabel {$K$} [bl] at 19 20
\endlabellist
\includegraphics[scale=1.0]{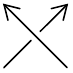}
\end{array}\ \right)
\end{equation}
\end{example}

\subsection{Statement of the generalized splitting formulas}

Let $f,g\geq 0$ be integers. It follows from Lemma \ref{lem:lk} and the equivalence (i) $\Leftrightarrow$ (iii)  in Lemma \ref{lem:Q-LP} 
that the $\Q$-LP surgery equivalence class of a cobordism $(M,m)\in \QLCob(g,f)$ is classified by the matrix $\Lk_B(\gamma)$,
where $(B,\gamma) \in \btT(g,f)$ is the corresponding bottom-top tangle. 
According to the decomposition \eqref{eq:sqcup}, this equivalence class is encoded by the ``strut'' part $\Ztilde^s(M)$ of $\Ztilde(M)$.

Thus, we  fix  in the sequel a $\Q$-LP surgery equivalence class $\mathcal{M} \subset \QLCob(g,f)$ 
and consider only the  ``Y'' part $\Ztilde^Y$  of the LMO functor.
Besides, we shall consider the following variants of $\Ztilde^Y: \mathcal{M} \to \A^Y(\set{g}^+ \cup \set{f}^-)$.
Set $\mathbf{B\!A} := \mathbf{B\!A}(g,f)$ and identify the vector spaces  $\A^Y(\mathbf{B\!A})$ and  $\A^Y(\set{g}^+ \cup \set{f}^-)$
by  the following correspondence of colors: 
$$
\xymatrix @!0 @R=0.5cm @C=4cm  {
\mathbf{B\!A} = B_g^\Q \oplus A_f^\Q \ar[r]^-\simeq  & \Q\! \cdot\! (\set{g}^+ \cup \set{f}^- ) \\
\beta_j \ar@{|->}[r] & j^+\\
\alpha_i \ar@{|->}[r] & i^-
}
$$
For any isotropic $\mathcal{M}$-essential subspace $\mathbf{E}$ of $H_1(\partial C^g_f;\Q)$, 
consider the isomorphism $\kappa_{\mathcal{M}}^{\mathbf{B\!A},\mathbf{E}}:  \A^Y(\mathbf{B\!A}) \to \A^Y(\mathbf{E})$ introduced in \S \ref{subsec:Q-LP} and set
$$
Z^{\mathbf{E}}_{\mathcal{M}} := \kappa_{\mathcal{M}}^{\mathbf{B\!A},\mathbf{E}} \circ \Ztilde^Y: \mathcal{M} \longrightarrow  \A^Y(\mathbf{E}).
$$

\begin{example} \label{ex:Z}
Take $\mathcal{M}:= \QCyl(F_g)$ which, by Lemma \ref{lem:Q-LP},  is the $\Q$-LP surgery equivalence class of the standard cylinder $F_g \times [-1,1]$.
Identify the space $\A^Y(\mathbf{D})$ with $\A^Y(H_\Q)$  
where  $\mathbf{D}$  is the $\mathcal{M}$-essential isotropic subspace discussed in Example \ref{ex:D} 
and is identified with $H_\Q = H_1(F_g;\Q)$ by the map $h \mapsto (h/2,h/2)$.
Let also $s:\A^Y(H_\Q) \to \A^Y(H_\Q)$ be the automorphism defined by $s(D):= (-1)^{\chi(D)}\! \cdot\! D$ 
where $\chi(D)$ is the Euler characteristic of an $H_\Q$-colored Jacobi diagram $D$.
It is easily checked that $s \circ \kappa_{\mathcal{M}}^{\mathbf{B\!A},\mathbf{D}}$ is  the isomorphism denoted by $\kappa$  in \cite[\S 4.1]{MM}. 
Hence $s \circ Z^{\mathbf{D}}_{\mathcal{M}}$ coincides with the \emph{LMO homomorphism} 
$$Z:=\kappa \circ  \ZtildeY: \QCyl(F_g) \longrightarrow \A^Y(H_\Q)$$
which has been studied in \cite{HabiroMassuyeau,MM}.
\end{example}

We can now state our generalized splitting formulas.

\begin{theorem}\label{th:general}
Let  $\mathcal{M} \subset \QLCob(g,f)$ be a $\Q$-LP surgery equivalence class and let $\mathbf{E}$ be an $\mathcal{M}$-essential isotropic subspace of $H_1(\partial C_f^g;\Q)$.
For any $M \in \mathcal{M}$ and for any finite  family $\mathsf{C}=(\mathsf{C}_1,\dots, \mathsf{C}_r)$  of $\Q$-LP pairs 
such that $C_i' \subset \operatorname{int}(M)$ and $C_i'\cap C_j' = \varnothing$ for all $i\neq j$,  
\begin{equation}\label{eq:general}
\sum_{I \subset \{1,\dots,r\} } (-1)^{|I|} \cdot Z^{\mathbf{E}}_{\mathcal{M}}  \left(M_{\mathsf{C}_I}\right)
= \rho_{\mathsf{C}}^{\mathbf{E}}\Bigg(\!\!\begin{array}{c} \hbox{\small sum of all ways of gluing  }\\
\hbox{\small  some legs of $\mu(\mathsf{C})$ with $\ell_M^{\mathbf{E}}(\mathsf{C})/2$} \end{array}\!\!\Bigg) + (\ideg >r).
\end{equation}
\end{theorem}

\noindent
Here we have implicitly fixed some  non-associative words $w$ and $v$ in the single letter $\bullet$ of length $g$ and $f$, respectively,
so that any $N\in \QLCob(g,f)$ is upgraded to  $N\in \QLCob_q(w,v)$. We have  also used the same notations as in the Introduction.
Thus, for any $I\subset \{1,\dots,r\}$,  $M_{\mathsf{C}_I}$  is the manifold obtained from $M$ by 
the $\Q$-LP surgeries $M \leadsto M_{\mathsf{C}_i}$ performed simultaneously for all $i\in I$;
the sum of Jacobi diagrams $\mu(\mathsf{C}) \in \A(H_1(C;\Q))$ and the symmetric bilinear form $\ell_M^{\mathbf{E}}(\mathsf{C}): H_1(C;\Q) \times H_1(C;\Q)  \to \Q$
are defined as in the Introduction except that, for the latter, we use the generalized linking number $\Lk_M^{m_*(\mathbf{E})}(-,-)$ in  $M$
instead of the usual linking number in a   $\Q$-homology $3$-sphere (see Lemma \ref{lem:P&Q}).
Finally, $\rho_{\mathsf{C}}^{\mathbf{E}}: \A(H_1(C;\Q)) \to \A(\mathbf{E})$ is the linear map that changes the colors  as follows:
\begin{equation} \label{eq:change_colors-general} 
\xymatrix {
H_1(C;\Q) & {\displaystyle \bigoplus_{i=1}^r H_1(C'_i;\Q)} \ar[r]^-{\incl_*} \ar[l]_-{\incl_*}^-\simeq  & H_1(M;\Q) & \mathbf{E} \ar[l]_-{ m_{*}}^-\simeq.
}
\end{equation}

Theorem \ref{th:general} implies the theorem stated in the Introduction.
Indeed, any $\Q$-homology $3$-sphere $S$ can be transformed into a  cobordism 
$M \in \QLCob(0,0)$ by removing an open $3$-ball. We take $\mathcal{M}:= \QLCob(0,0)$, 
which is the $\Q$-LP surgery equivalence class of the standard sphere $S^3$, and $\mathbf{E}:=0$.
Then the linear map $\rho_{\mathsf{C}}^{\mathbf{E}}:\A(H_1(C;\Q)) \to \A(\varnothing)$
kills any diagram having at least one univalent vertex, so that the ``sum of all ways of gluing some legs''
turns into a ``sum of all ways of gluing all legs.'' 

\begin{example}
Take $\mathcal{M}:= \QCyl(F_g)$ and $\mathbf{E} :=\mathbf{D}$ as in Example  \ref{ex:Z}.
In this case,  Theorem \ref{th:general} provides splitting formulas  for the LMO homomorphism $Z: \QCyl(F_g) \to \A^Y(H_\Q)$
which involve the notion of linking number  discussed in Example \ref{ex:D}.
For $r=2$, these formulas are generalizations of   \cite[Proposition C.2]{MM}.
\end{example}

A consequence of Theorem \ref{th:general} is that $Z^{\mathbf{E}}_{\mathcal{M}}$ is universal 
among $\Q$-valued finite-type invariants (in the sense of Ohtsuki and Goussarov--Habiro).
This universality property  is obtained in \cite[\S 7]{CHM} in the special case $\mathbf{E}=\mathbf{B\!A}$ and for  $\Z$-Lagrangian cobordisms.
Theorem \ref{th:general} is proved in the next section by enhancing the arguments of \cite[Theorem 7.11]{CHM}.

\section{Proof of the generalized splitting formulas}  \label{sec:proof}

This section is aimed at proving Theorem \ref{th:general}. 

\subsection{A special case}

The following is the specialization of Theorem \ref{th:general} for  $\mathbf{E}=\mathbf{B\!A}$.

\begin{lemma}\label{lem:special}
Let  $w$ and $v$ be non-associative words in the single letter $\bullet$
of length $g$ and $f$ respectively.
Let  $M\in \QLCob_q(w,v)$ 
and let $\mathsf{C}=(\mathsf{C}_1,\dots, \mathsf{C}_r)$ be a family of  $\Q$-LP pairs 
where $C_i' \subset \operatorname{int}(M)$ and $C_i'\cap C_j' = \varnothing$ for all $i\neq j$.
Then we have
\begin{equation}\label{eq:special}
\sum_{I \subset \{1,\dots,r\} } (-1)^{|I|} \cdot \ZtildeY\left(M_{\mathsf{C}_I}\right)
= \rho_{\mathsf{C}}\Bigg(\!\!\begin{array}{c} \hbox{\small sum of all ways of gluing  }\\
\hbox{\small  some legs of $\mu(\mathsf{C})$ with $\ell_M(\mathsf{C})/2$} \end{array}\!\!\Bigg) + (\ideg >r).
\end{equation}
\end{lemma}

\noindent
Here $\ell_M(\mathsf{C}):= \ell_M^{\mathbf{B\!A}}(\mathsf{C})$ and the map
$\rho_{\mathsf{C}}:= \rho_{\mathsf{C}}^{\mathbf{B\!A}}: \A(H_1(C;\Q)) \to \A(\mathbf{B\!A}) \simeq \A(\set{g}^+ \cup \set{f}^-)$ 
consists in changing  the colors of univalent vertices as follows:
\begin{equation}\label{eq:change_colors}
\xymatrix @!0 @R=0.5cm @C=32mm  {
H_1(C;\Q) & {\displaystyle \bigoplus_{i=1}^r H_1(C'_i;\Q)} \ar[r]^-{\incl_*} \ar[l]_-{\incl_*}^-\simeq  
& H_1(M;\Q) & B_g^\Q \oplus A_f^\Q \ar[l]_-{ m_{+,*}\oplus m_{-,*}}^-\simeq  
}
\end{equation}

Using the notations of Lemma \ref{lem:special}, we now  show that formula \eqref{eq:special} implies \eqref{eq:general}
for any  isotropic  $M$-essential subspace $\mathbf{E}$ of $H_1(\partial C_f^g;\Q)$.
Let $D \in \A(H_1(C;\Q))$ be a Jacobi diagram of the form
$$
D= \Ygraphtop{c_1}{b_1}{a_1} \sqcup \cdots \sqcup \Ygraphtop{c_r}{b_r}{a_r}
$$
where $a_1,b_1,c_1 \in H_1(C_1;\Q),\dots, a_r,b_r,c_r \in H_1(C_r;\Q)$.
Consider $k$ pairs of univalent vertices  $\{v_1,w_1\},\dots, \{v_k,w_k\}$  of $D$
such that  $\{v_i,w_i\}\cap \{v_j,w_j\} = \varnothing$ for any $i\neq j$. 
By making the identifications $v_i\equiv w_i$ for every $i\in \{1,\dots,k\}$ and by applying the map \eqref{eq:change_colors-general} to the colors of the remaining univalent vertices,
we obtain an $\mathbf{E}$-colored Jacobi diagram $D'$.  
We can assume that $v_i$ and $w_i$ belong to different connected components of $D$ for all $i\in \{1,\dots,k\}$ 
because  $D'$ is otherwise trivial in $\A(\mathbf{E})$ by the AS relation.
For all $i\in \{1,\dots,k\}$, let $V_i,W_i \subset C'_1 \sqcup \cdots \sqcup C'_r \subset M$ be oriented knots representing 
$$
\col(v_i),\col(w_i) \in H_1(C;\Q) \simeq H_1(C'_1;\Q) \oplus \cdots \oplus H_1(C'_r;\Q),
$$
respectively, and let $f$ be the inverse of $m_*\!:\! \mathbf{BA} \to H_1(M;\Q)$.
The  coefficient of $D'$ in 
$$
\rho_{\mathsf{C}}^{\mathbf{E}}\left(\!\!\begin{array}{c} \hbox{\small sum of all ways of gluing  }
\hbox{\small  some legs of $D$ with $\ell_M^{\mathbf{E}}(\mathsf{C})/2$} \end{array}\!\!\right)
$$
 is the product 
\begin{eqnarray*}
\prod_{j=1}^k  \Lk_M^{m_*(\mathbf{E})}\left( V_j , W_j\right) 
& \stackrel{\eqref{eq:variation_Lk}}{=}&\prod_{j=1}^k  \left( \Lk_M(V_j, W_j) + \vartheta_M^{m_*(\mathbf{E})}\big(m_*f([V_j]), m_*f([W_j])\big)\right) \\
&=&  \sum_{P\subset \{1,\dots,k\}}\  \prod_{p\in P} \Lk_M(V_p,W_p) \cdot  \prod_{q\not\in P}  \vartheta_{\mathcal{M}}^{\mathbf{E}}\big(f([V_q]), f([W_q])\big).
\end{eqnarray*}
This is also the coefficient of $D'$ in 
$$
\kappa_{\mathcal{M}}^{\mathbf{B\!A},\mathbf{E}} \rho_{\mathsf{C}}\left(\!\!\begin{array}{c} \hbox{\small sum of all ways of gluing  }
\hbox{\small  some legs of $D$ with $\ell_M(\mathsf{C})/2$} \end{array}\!\!\right)
$$
since the colors of $v_i$ and $w_i$ are mapped by \eqref{eq:change_colors} to $f([V_i])$ and $f([W_i])$ respectively.
So \eqref{eq:general} directly follows from \eqref{eq:special} by applying $\kappa_{\mathcal{M}}^{\mathbf{B\!A},\mathbf{E}}$.
The proof of  formula \eqref{eq:special} is postponed to \S \ref{subsec:proof}. Before that, we need to establish a few technical results.

\subsection{From $\Q$-homology handlebodies to $\Q$-Lagrangian cobordisms}

The following will be useful to find appropriate boundary parameterizations of $\Q$-homology handlebodies.

\begin{lemma}
\label{lem:realizing_Lagrangian}
Let $\Sigma$ be a compact connected oriented surface of genus $g$ with $\partial \Sigma \cong S^1$,
and let $L^\Q$ be a Lagrangian subspace of $H_1(\Sigma;\Q)$.
Then there exists an orientation-preserving homeomorphism $s:  F_g  \to \Sigma$ 
such that $s_*(\alpha_1), \dots , s_*(\alpha_g)$ span $L^\Q$.
\end{lemma}

\begin{proof}
Let $L := L^\Q \cap H_1(\Sigma;\Z)$. 
Then $L$ is an isotropic subgroup of $H_1(\Sigma;\Z)$,
its rank is $\dim_\Q(L\otimes \Q) = \dim_\Q(L^\Q) = g$
and  the quotient $H_1(\Sigma;\Z)/L$ is torsion-free.
(Therefore, $L$ is a Lagrangian subgroup of $H_1(\Sigma;\Z)$, i.e$.$ it is isotropic maximal.)

Let $m=(m_1,\dots, m_g)$ be a basis of the free abelian group $L$. 
Since $H_1(\Sigma;\Z)/L$ is free, we can complement $m$ to a basis $(m,p)$ of $H_1(\Sigma;\Z)$. 
The matrix of the intersection pairing of $\Sigma$ in the basis $(m,p)$ is of the form
$$
\left( \begin{array}{cc}
0 & P^t \\ -P & Q
\end{array} \right)
$$
where $P$ is unimodular and $Q$ is antisymmetric. 
We write $Q$ as $R - R^t$ where $R$ is lower triangular, and we observe that
$$
\left( \begin{array}{cc}
I & 0 \\  P^{-1} R^t (P^{-1})^{t} & P^{-1}
\end{array} \right)
\cdot \left( \begin{array}{cc}
0 &  P^t \\ -P & Q
\end{array} \right) \cdot  
\left( \begin{array}{cc}
I &   P^{-1} R (P^{-1})^t \\ 0 & (P^{-1})^t
\end{array} \right)
= \left( \begin{array}{cc}
0 & I \\ -I & 0
\end{array} \right).
$$
Thus, we can complement $m$ to another basis $(m,p')$ of $H_1(\Sigma;\Z)$
with respect to which the matrix of the intersection pairing is
$$
\left( \begin{array}{cc} 0 & I \\ -I & 0 \end{array} \right).
$$ 
Then the  isomorphism $H_1(F_g;\Z) \to H_1(\Sigma;\Z)$ defined by 
$\alpha_i \mapsto m_i$ and $\beta_i \mapsto p'_i$ preserves the intersection pairing.
So, by a classical theorem of Dehn and Nielsen, we can realize this isomorphism
by an orientation-preserving homeomorphism $s:F_g \to \Sigma$. The conclusion follows since, by construction, we have
$$
\big\langle s_*(\alpha_1), \dots , s_*(\alpha_g) \big\rangle = \langle m_1,\dots, m_g \rangle 
= L \ \subset H_1(\Sigma;\Z).
$$

\up
\end{proof}

We now explain how to turn $\Q$-homology handlebodies into $\Q$-Lagrangian cobordisms.

\begin{lemma}
\label{lem:QHH_to_QLCob}
Let $C'$ be a $\Q$-homology handlebody of genus $g$. Then there exists
a boundary parameterization $c': \partial C^{g}_0 \to  C'$ such that $(C',c') \in \QLCob(g,0)$.
\end{lemma}

\begin{proof}
Let $\Sigma$ be the surface obtained from $\partial C'$ by removing an open disk.
The Lagrangian of $C'$, namely
$$
\mathbf{L}_{C'}^\Q =\Ker \left(\incl_*: H_1(\partial C';\Q) \longrightarrow H_1(C';\Q) \right),
$$ 
is a Lagrangian subspace of 
$H_1(\partial C';\Q)$ $\simeq H_1(\Sigma;\Q)$. So, by Lemma \ref{lem:realizing_Lagrangian}, 
we can find an orientation-preserving homeomorphism $s: F_{g} \to \Sigma$ 
such that $s_*(A_{g}^\Q) = \mathbf{L}_{C'}^\Q$. 
Next, seeing $\partial C^{g}_0$ as the union of $F_{g}$ with a closed disk,
we can extend $s$  to an orientation-preserving homeomorphism $\partial C^{g}_0 \to \partial C'$
which, together with the inclusion $\partial C' \subset C'$, defines a boundary parameterization $c':\partial C^{g}_0 \to C'$.
The cobordism $(C',c')$ is $\Q$-Lagrangian because
\begin{itemize}
\item[(1)] $H_1(C';\Q) =  0+ \incl_*\left( H_1(\Sigma;\Q) \right)
= c'_{-,*}(A_0^\Q) +  c_{+,*}'\! \left(H_1(F_{g};\Q)\right)$, 
\item[(2)] $c'_{+,*}(A_{g}^\Q) = \incl_*\circ  s_* (A_{g}^\Q) = \incl_*\big(\mathbf{L}_{C'}^\Q\big) \subset 0= c'_{-,*}(A_0^\Q).$
\end{itemize}
(For  (1), we have used the fact that $\incl_*:H_1(\partial C';\Q) \to H_1(C';\Q)$ is surjective.)
\end{proof}

\subsection{The ``Y'' part of the LMO functor}

The next lemma relates triple-cup products of closed oriented $3$-manifolds
to $\Ztilde_1$, the  i-degree one part of the LMO functor.

\begin{lemma}
\label{lem:triple-cup}
Let $\mathsf{C}=(C',C'')$ be a $\Q$-LP pair of genus $g$. 
Consider some boundary parameterizations
$$
c': \partial C^g_0 \longrightarrow  C', \quad
c'': \partial C^g_0 \longrightarrow C''
$$
that are compatible with the given identification $\partial C' \equiv \partial C''$ and that satisfy
$(C',c')\in \QLCob(g,0)$ and $(C'',c'') \in \QLCob(g,0)$.
Then, the triple-cup product form of the total manifold $C= (-C') \cup_\partial C''$ is given by
$$
\mu(C) = \Ztilde_1(C',c') - \Ztilde_1(C'',c'') \ \in  \AY_1(\set{g}^+) \simeq \Lambda^3 H_1(C;\Q).
$$
\end{lemma}

\noindent
Here, the isomorphism between ${\AY_1(\set{g}^+)}$ and ${\Lambda^3 H_1(C;\Q)}$ is defined by 
\begin{equation}
\label{eq:Y_to_trivector}
\Ygraphtop{k^+}{j^+}{i^+} \longmapsto [c'_{+}(\beta_i)]  \wedge [c'_{+}(\beta_j)]  \wedge [c'_{+}(\beta_k)].
\end{equation}
Besides, the $\Q$-Lagrangian cobordisms $(C',c')$ and $(C'',c'')$ are equipped
with  an arbitrary non-associative word of  length $g$ in the single letter $\bullet$.

\begin{proof}[Proof of Lemma \ref{lem:triple-cup}]
The existence of the boundary parameterizations $c'$ and $c''$ 
follows from Lemma  \ref{lem:QHH_to_QLCob}.
Each of the closed $3$-manifolds 
$$
(-C_0^g) \cup_{c'} C', \quad (-C_0^g) \cup_{c''} C'' \quad \hbox{and} \quad C = (-C') \cup_\partial C'' 
$$
is contained in the singular $3$-manifold
$
(-C_0^g) \cup_{c' \sqcup c''} (C' \sqcup C'').
$
By computing triple-cup products in this topological space, we find that
$$
\mu\left((-C_0^g) \cup_{c''} C'' \right) - \mu\left((-C_0^g) \cup_{c'} C' \right)
= \mu(C) \ \in \ \Lambda^3 H_1(C;\Q)
$$
where the homology groups are identified 
through the following isomorphisms (which are all induced by inclusions):
$$
\left\{ \begin{array}{l}
\xymatrix{H_1\left((-C_0^g) \cup_{c'} C' ;\Q \right) & H_1(C';\Q) \ar[l]_-{\simeq} \ar[r]^-{\simeq} & H_1(C;\Q)} \\
\xymatrix{H_1\left((-C_0^g) \cup_{c''} C'' ;\Q \right) & H_1(C'';\Q) \ar[l]_-{\simeq} \ar[r]^-{\simeq} & H_1(C;\Q)}
\end{array} \right.
$$
Thus it suffices to show that 
$\mu\left((-C_0^g) \cup_{c'} C' \right) = - \Ztilde_1(C',c')$,
and similarly for $(C'',c'')$.

In order to prove this identity, we consider the  bottom-top tangle $(B,\gamma)$ 
corresponding to the cobordism $(C',c')\in \Cob(g,0)$ by the correspondence \eqref{eq:btT_to_Cob}.
In this case, there is no bottom component in $\gamma$ (i.e. $\gamma=\gamma^+$)
and, since $(C',c')$ is $\Q$-Lagrangian, we have 
$$
\Lk_B(\gamma) = \Lk_{\hat B}(\hat \gamma) = 0
$$
where $\hat \gamma \subset \hat B$ is the plat closure of $\gamma$ in the $\Q$-homology $3$-sphere  $\hat B$.
Recall  that $\Ztilde(C',c')$ is defined in \cite{CHM} as a certain renormalization 
of the Kontsevich--LMO invariant $\chi^{-1}Z(B,\gamma)$;
however this renormalization does not affect the ``Y'' part. 
Therefore, $\Ztilde_1(C',c')$ is the i-degree $1$ part of $\chi^{-1}Z(B,\gamma)$.
Since the non-diagonal coefficients of $\Lk_B(\gamma)$ are trivial,
the oriented link $\hat{\gamma}$ is  algebraically-split.
Hence we can apply Lemma \ref{lem:triple_Milnor_Y} 
to  deduce that $-\Ztilde_1(C',c')$ is the linear combination of Y-shaped diagrams
encoding Milnor's triple linking numbers of $\hat{\gamma}$ in $\hat{B}$.
The closed oriented $3$-manifold $(-C_0^g) \cup_{c'} C'$ 
is obtained from $\hat{B}$ by surgery along the framed link $\hat{\gamma}$.
Since the $i$-th diagonal coefficient of $\Lk_B(\gamma)$ 
-- i.e$.$ the framing number of $\hat \gamma_i$ -- is trivial, 
the parallel of the framed knot $\hat \gamma_i$   
is also the longitude of $\hat \gamma_i$ in the $\Q$-homology $3$-sphere $\hat B$:
therefore, the surgery along the framed link $\hat{\gamma}$ is a longitudinal surgery.
Remembering now the exact connection \eqref{eq:mu} between Milnor's triple linking numbers and triple-cup products, 
we conclude that  $\mu\left((-C_0^g) \cup_{c'} C' \right) = - \Ztilde_1(C',c')$. 
The same conclusion applies to $(C'',c'')$ with the same arguments.
\end{proof}

\subsection{Proof of the special case}   \label{subsec:proof}

We can now prove Lemma \ref{lem:special}, which will finish the proof of Theorem \ref{th:general}.
Let  $w$ and $v$ be non-associative words in the single letter $\bullet$ of length $g$ and $f$ respectively.
Let  $(M,m)\in \QLCob_q(w,v)$ 
and let $\mathsf{C}=(\mathsf{C}_1,\dots, \mathsf{C}_r)$ be a family of  $\Q$-LP pairs 
where $C_i' \subset \operatorname{int}(M)$ and $C_i'\cap C_j' = \varnothing$ for all $i\neq j$.

We denote by $e_1,\dots,e_r$ the genus of the $\Q$-homology handlebodies $C'_1,\dots,C'_r$ respectively.
We apply Lemma \ref{lem:QHH_to_QLCob} to each of $C'_1,\dots,C'_r$ 
and find boundary parameterizations $c'_1,\dots,c'_r$ such that
$$
(C'_1,c'_1) \in \QLCob(e_1,0),\dots,(C'_r,c'_r) \in \QLCob(e_r,0).
$$
For any $i\in \{1,\dots,r\}$, let $c''_i:\partial C_0^{e_i} \to C''_i$ be the boundary parameterization corresponding 
to $c'_i$ through the identification $\partial C'_i \equiv \partial C''_i$. Then, by definition of a $\Q$-LP pair, we also have
$$
(C''_1,c''_1) \in \QLCob(e_1,0),\dots,(C''_r,c''_r) \in \QLCob(e_r,0).
$$

We consider next a collar neighborhhood  $m_-(F_f)\times [-1,0]$  in $M$ of the bottom surface $m_-(F_f)\equiv m_-(F_f) \times \{-1\}$, 
we pick $r$ pairwise-disjoint closed disks on  $m_-(F_f) \times \{0\}$
and we connect them  to the disks $c'_{1,-}(F_0) \subset \partial C'_1, \dots, c'_{r,-}(F_0) \subset \partial C'_r$ 
by pairwise-disjoint solid tubes $T_1,\dots, T_r$ in the exterior of $(m_-(F_f)\times [-1,0]) \cup C'_1 \cup \cdots \cup C'_r$.
Thus we obtain a decomposition of $(M,m)$ in the monoidal category $\Cob$:
\begin{equation}\label{eq:decomposition}
(M,m) = \left( (C'_1,c'_1) \otimes \cdots \otimes (C'_r,c'_r) \otimes \Id_f\right) \circ (N,n)
\end{equation}
where $\Id_f$ denotes the identity of $f$ in $\Cob$ and $(N,n) \in \Cob(g,e+f)$ with $e:= e_1 + \cdots + e_r$.
(Here $N$ corresponds to the exterior in $M$ of the union of 
$(m_-(F_f)\times [-1,0])$,  $C'_1 \cup \cdots \cup C'_r$ and  $T_1 \cup \cdots \cup T_r$.)
In fact, we have $(N,n) \in \QLCob(g,e+f)$ so that \eqref{eq:decomposition} is actually 
a decomposition in the subcategory $\QLCob$.
To check this, we consider the bottom-top tangles 
$(B,\gamma)$ and $(D,\upsilon)$ corresponding to $M$ and $N$ respectively.
Then
$$
D = C_0^{e+f} \circ N \circ C_g^0 = \left((C_0^{e_1} \otimes \cdots \otimes C_0^{e_r}) \otimes C_0^f\right)\circ N   \circ C_g^0
$$
can be obtained by $r$ $\Q$-LP surgeries from
$$
\left((C'_1\otimes \cdots \otimes C'_r) \otimes C_0^f\right)\circ N   \circ C_g^0 
\stackrel{\eqref{eq:decomposition}}{=} C_0^f \circ M \circ C_g^0 = B,
$$
and these surgeries transform the top tangle $\upsilon^+ \subset D$ into the top tangle $\gamma^+ \subset B$.
The cobordism $M$ being $\Q$-Lagrangian, $B$ is a $\Q$-homology cube and $\Lk_B(\gamma^+)=0$.
Since  any $\Q$-LP surgery preserves the $\Q$-homology type as well as linking numbers (see Lemma \ref{lem:pre_Q-LP}), 
we deduce that $D$ is a $\Q$-homology cube and $\Lk_D(\upsilon^+)=0$. 
Hence the cobordism $N$ is $\Q$-Lagrangian.

In order to apply the LMO functor,
we  choose for every $i\in \{1,\dots, r\}$ a non-associative word $u_i$ of length $e_i$ in the single letter $\bullet$,
with which we equip the $\Q$-Lagrangian cobordisms $(C'_i,c'_i)$ and $(C''_i,c''_i)$.
The $\Q$-Lagrangian cobordism $(N,n)$ is  equipped with the non-associative word
$(\cdots ((u_1u_2)u_3)\cdots u_r)v$. 
Then \eqref{eq:decomposition} is refined to the following decomposition in the category $\QLqCob$: 
$$
M  = \left( \Big( \cdots \big( ( C'_1 \otimes C'_2 ) \otimes C'_3 \big)
  \cdots \otimes C'_r \Big) \otimes \Id_v \right) \circ N.
$$
For every subset $I \subset \{1,\dots,r\}$, we have the same formula for the cobordism $M_{\mathsf{C}_I}$
except that  $C'_i$ is now replaced by $C''_i$ for all $i\in I$. 
Then, by applying the tensor-preserving functor $\Ztilde$, we obtain that
$$
\sum_{I \subset \{1,\dots,r\} } (-1)^{|I|} \cdot \Ztilde(M_{\mathsf{C}_I})
= \sum_{I \subset \{1,\dots,r\} } (-1)^{|I|}  \cdot
\left( \Ztilde(C^{?}_1) \otimes  \cdots \otimes \Ztilde(C^?_r) \otimes \Id_f \right) \circ \Ztilde(N)
$$
where each question mark should be replaced by a prime or a double prime (depending on the subset $I$).
Using the bilinearity of $\circ$ and $\otimes$ in the category $\tsA$, we deduce that
$$
\sum_{I \subset \{1,\dots,r\} } (-1)^{|I|} \cdot \Ztilde(M_{\mathsf{C}_I})
=  \Big(\big(\Ztilde(C'_1) - \Ztilde(C''_1)\big) \otimes \cdots \otimes 
\big(\Ztilde(C'_r) - \Ztilde(C''_r)\big)\otimes \Id_f \Big) \circ \Ztilde(N).
$$
and, using Lemma \ref{lem:triple-cup}, we get
$$
\sum_{I \subset \{1,\dots,r\} } (-1)^{|I|} \cdot \Ztilde(M_{\mathsf{C}_I}) = 
(\mu(C_1) \otimes \cdots \otimes \mu(C_r) \otimes \Id_f)\circ \Ztilde(N) + (\ideg > r)
$$
where each $\mu(C_i)$ is regarded as an element of $\A^Y(\set{e_i}^+) \subset \tsA(e_i,0)$ by \eqref{eq:Y_to_trivector}.
For the sequel, it will be convenient to decompose the set of colors $\set{e}^+=\{1^+,\dots,e^+\}$ into
$$
\underbrace{\{1^+,\dots ,e_1^+\}}_{E_1} \cup \underbrace{\{(e_1+1)^+,\dots ,(e_1+e_2)^+\}}_{E_2} \cup \cdots
\cup \underbrace{\left\{\left(\sum_{i=1}^{r-1}e_i+1\right)^+,\dots, \left(\sum_{i=1}^{r-1}e_i + e_r \right)^+ \right\}}_{E_r}.
$$
Thus, by definition of the tensor product in $\tsA$, we obtain that
\begin{equation}
\label{eq:decomposition_bis}
\sum_{I \subset \{1,\dots,r\} } (-1)^{|I|} \cdot \Ztilde(M_{\mathsf{C}_I}) \stackrel{r}{\equiv}
\Big(\mu(C_1) \sqcup \cdots \sqcup \mu(C_r) 
\sqcup \exp_\sqcup\big(\sum_{i=e+1}^{e+f} \strutgraph{i^-}{i^+}\big)\Big)
\circ \Ztilde(N) 
\end{equation}
where the symbol $\stackrel{r}{\equiv}$ means an identity modulo terms of i-degree $>r$
and each $\mu(C_i)$ is now regarded as an element of $\A^Y(E_i)$.

To proceed, we use again 
the bottom-top tangles $(B,\gamma)$ and $(D,\upsilon)$ corresponding to the cobordisms $M$ and $N$, respectively.
According to \eqref{eq:sqcup}, the ``strut'' part of $\Ztilde(N)$ is $\exp_\sqcup(\Lk_D(\upsilon)/2)$
so that  \eqref{eq:decomposition_bis} simplifies to 
\begin{equation}
\label{eq:decomposition_ter}
\sum_{I \subset \{1,\dots,r\} } (-1)^{|I|} \cdot \Ztilde(M_{\mathsf{C}_I}) \stackrel{r}{\equiv}
\Big(\mu(C_1) \sqcup \cdots \sqcup \mu(C_r) 
\sqcup \exp_\sqcup\big(\sum_{i=e+1}^{e+f} \strutgraph{i^-}{i^+}\big)\Big)
\circ \exp_\sqcup(\Lk_D(\upsilon)/2).
\end{equation}
We denote by $(\delta_1,\dots, \delta_r)$ the first $e$ components of the $(e+f)$-component tangle $\upsilon^-$. 
As we observed above, $D$ can be obtained from $B$ by $\Q$-LP surgeries: 
the tangles $\upsilon^+$ and $\upsilon^-\setminus \delta$ 
correspond through these surgeries to $\gamma^+$ and $\gamma^-$, respectively.
Since a $\Q$-LP surgery preserves linking numbers,
the symmetric matrix $\Lk_D(\upsilon)$  -- whose rows and columns are indexed by 
$\pi_0(\upsilon^+)\cup \pi_0(\delta) \cup  \pi_0(\upsilon^-\setminus \delta)$ --
can be decomposed as follows:
$$
\Lk_D(\upsilon) = \left(\begin{array}{ccc}
0&  \Lk_D(\upsilon^+,\delta)  &  \Lk_B(\gamma^+,\gamma^-) \\
\Lk_D(\delta,\upsilon^+)   &  \Lk_D(\delta) &   \Lk_D(\delta,\upsilon^-\setminus \delta) \\
 \Lk_B(\gamma^-,\gamma^+)  &\Lk_D(\upsilon^-\setminus \delta ,\delta) & \Lk_B(\gamma^-) 
\end{array}\right)
$$
Observe that the corner blocks of that matrix constitute $\Lk_B(\gamma)$.
Next, there exist simple combinatorial rules to compute 
compositions in the category $\tsA$ of the form $$(\exp_{\sqcup }(H/2) \sqcup  h^Y) \circ (\exp_{\sqcup }(J/2) \sqcup  j^Y),$$
where $H,J$ are rational matrices (interpreted as linear combinations of struts)
and $h^Y,j^Y$ have no strut component: see \cite[Lemma 4.5]{CHM}.
Applying these formulas to the right-hand side of \eqref{eq:decomposition_ter}, we obtain 
$$
\sum_{I \subset \{1,\dots,r\} } (-1)^{|I|} \cdot \Ztilde(M_{\mathsf{C}_I}) \stackrel{r}{\equiv}
\exp_\sqcup\left(\Lk_B(\gamma)/2\right)
\sqcup 
\tilde\rho_{\mathsf{C}}\Bigg(\!\!\begin{array}{c} \hbox{\small sum of all ways of gluing  }\\
\hbox{\small  some legs of $\mu(\mathsf{C})$ with $\Lk_D(\delta)/2$} \end{array}\!\!\Bigg)
$$
where $\mu(\mathsf{C}) = \mu(C_1) \sqcup \cdots \sqcup \mu(C_r)$ is regarded as an element of $\A^Y(\set{e}^+)$,
the symmetric matrix $\Lk_D(\delta)/2$ is seen as a symmetric bilinear form on the vector space $\Q \cdot \set{e}^+$ and 
$\tilde\rho_{\mathsf{C}}:\A^Y(\set{e}^+) \to \A^Y(\set{g}^+\cup \set{f}^-)$ changes the colors as follow:
\begin{equation}\label{eq:change_colors_bis}
\forall l^+\in \set{e}^+, \quad l^+ \longmapsto  \sum_{j=1}^g \Lk_D(\upsilon^+_j, \delta_l)\cdot  j^+ +
\sum_{i=1}^f \Lk_D\left(\upsilon^-_{e+i},\delta_l\right) \cdot i^-.
\end{equation}
Since the strut part $\Ztilde^s$ of $\Ztilde$ is preserved under $\Q$-LP surgery, we obtain that
\begin{equation}\label{eq:final}
\sum_{I \subset \{1,\dots,r\} } (-1)^{|I|} \cdot \ZtildeY(M_{\mathsf{C}_I}) \stackrel{r}{\equiv}
\tilde\rho_{\mathsf{C}}\Bigg(\!\!\begin{array}{c} \hbox{\small sum of all ways of gluing  }\\
\hbox{\small  some legs of $\mu(\mathsf{C})$ with $\Lk_D(\delta)/2$} \end{array}\!\!\Bigg).
\end{equation}

It remains  to relate $\tilde \rho_{\mathsf{C}}$ to $\rho_{\mathsf{C}}$, and $\Lk_D(\delta)$ to  $\ell_M(\mathsf{C})$.
For the first relation,  observe that we have the following identity in $H_1(M;\Q)$
for all $k\in \{1,\dots, r\}$  and for all $l\in \{1,\dots,e_k\}$, where we set $\bar l := l + \sum_{s=1}^{k-1} e_s \in \{1,\dots,e\}$:
\begin{eqnarray*}
\incl_*([c'_{k,+}(\beta_l)]) &= &
\sum_{j=1}^g \Lk_B\big(c'_{k,+}(\beta_l), \gamma_j^+\big) \cdot [m_+(\beta_j)] 
+ \sum_{i=1}^f   \Lk_B\big(c'_{k,+}(\beta_l), \gamma_i^-\big)  \cdot [m_-(\alpha_i)] \\
&=& \sum_{j=1}^g \Lk_D\big(n_{-}(\beta_{\bar l}), \upsilon_j^+\big) \cdot [m_+(\beta_j)] 
+ \sum_{i=1}^f   \Lk_D\big(n_{-}(\beta_{\bar l}), \upsilon_{e+i}^-\big)  \cdot [m_-(\alpha_i)] \\
&=&  \sum_{j=1}^g \Lk_D\big(\delta_{\bar l}, \upsilon_j^+\big) \cdot [m_+(\beta_j)] 
+ \sum_{i=1}^f   \Lk_D\big(\delta_{\bar l}, \upsilon_{e+i}^-\big)  \cdot [m_-(\alpha_i)].
\end{eqnarray*}
Thus, by comparing \eqref{eq:change_colors_bis} to \eqref{eq:change_colors},
we see that $\tilde \rho_{\mathsf{C}}$ corresponds to $\rho_{\mathsf{C}}$ 
through the  isomorphism $\A^Y(\set{e}^+) \simeq \A(H_1(C;\Q))$ 
defined by the change of colors $\bar l^+ \mapsto [c'_{k,+}(\beta_l)]$ for all $k\in \{1,\dots, r\}$  and for all $l\in \{1,\dots,e_k\}$.
We now relate $\Lk_D(\delta)$ to  $\ell_M(\mathsf{C})$ and, for this,
we use the following notation: 
for any two colors $a,b \in \set{e}^+$,  we shall write $a \sim b$
if and only if $a,b$ belong to the same subset $E_k$  for a $k\in \{1,\dots,r\}$. 
If $a \sim b$, then gluing an $a$-colored  vertex to a $b$-colored  vertex in $\mu(\mathsf{C})$ 
does not contribute to the right-hand side term of \eqref{eq:final} due to the AS relation.
Consider the case $a\nsim b$: we assume that $a=\bar u\in E_x$ and $b = \bar v\in E_y$ with $x\neq y$, $u\in \set{e_x}$ and $v \in \set{e_y}$.
Then we have
\begin{eqnarray*}
\Lk_D(\delta_a,\delta_b) 
\ = \ \Lk_{D}(\hbox{\small parallel of } \delta_a,\hbox{\small parallel of } \delta_b) &=&
\Lk_{D} \big(n_-(\beta_a),n_-(\beta_b)\big)\\
&=& \Lk_{B} \big(n_-(\beta_a),n_-(\beta_b)\big) \\
&= &\Lk_{M}\left(c'_{x,+}(\beta_u), c'_{y,+}(\beta_v)\right),
\end{eqnarray*}
where the last equality follows from Lemma \ref{lem:lk}. 
Thus the non-diagonal blocks of the matrix of the pairing $\ell_{ M}(\mathsf{C}):H_1({C};\Q) \times H_1({C};\Q) \to \Q$ in the basis 
$$\big(c'_{1,+}(\beta_1),\dots,c'_{1,+}(\beta_{e_1}),\dots, c'_{r,+}(\beta_1),\dots,c'_{r,+}(\beta_{e_r})\big)$$ 
are  the non-diagonal blocks of $\Lk_D(\delta)$.
This concludes the proof of Lemma \ref{lem:special}.

\appendix

\section{Linking numbers in $\Q$-homology handlebodies}     \label{sec:lk}

In this appendix, we recall the definition of linking numbers in a $\Q$-homology handlebody,
we study the ambiguity inherent to this definition and we prove a few properties which are needed to establish the splitting formulas in their full generality.

\subsection{Definition}  \label{subsec:def_lk}

We mainly follow Cimasoni \& Turaev \cite{CT}.
Let $M$ be a $\Q$-homology handlebody with Lagrangian  $\mathbf{L}:= \mathbf{L}_M^\Q \subset H_1(\partial M;\Q)$.  A subspace $\mathbf{E}$ of  $H_1(\partial M;\Q)$ is said to be \emph{essential} 
if the restriction of $\incl_*:H_1(\partial M;\Q) \to H_1(M;\Q)$ to $\mathbf{E}$ is an isomorphism onto $H_1(M;\Q)$ or, equivalently, if $H_1(\partial M;\Q) = \mathbf{L} \oplus \mathbf{E}$;
the corresponding section of $\incl_*$  is denoted by the lower-case letter $e: H_1(M;\Q) \to H_1(\partial M;\Q)$.

As explained in \cite[\S 1]{CT}, any essential subspace $\mathbf{E}$ of  $H_1(\partial M;\Q)$ defines a notion of ``generalized linking number'' in $M$.
Specifically, the \emph{linking number}   of  two disjoint, oriented knots $K,L\subset  \operatorname{int}(M)$ 
is the unique number $\Lk_M^{\mathbf{E}}(K,L)\in \Q$ satisfying
\begin{equation} \label{eq:lk^S}
[L] = \Lk_M^{\mathbf{E}}(K,L) \cdot [m_K] +   \incl_* e([L]) \in H_1(M\setminus K;\Q).
\end{equation}
Here $m_K$ is the oriented meridian of $K$, $\incl_*$ is  induced by the inclusion $\partial M \subset  M\setminus K$ 
and we use the following fact:  the long exact sequence in homology for the pair $(M,M\setminus K)$ gives a short exact sequence
$$
\xymatrix{
0 \ar[r] & \Q [m_K] \ar[r] & H_1\left(M\setminus K;\Q\right) \ar[r]^-{\incl_*} & H_1(M;\Q) \ar[r] & 0.
}
$$
The generalized linking number can be computed by the formula
\begin{equation} \label{eq:lk_as_intersection}
\Lk_M^{\mathbf{E}}(K,L)  = K \mediumdot_{\!M}\, D,
\end{equation}
where $D$ is a rational $2$-chain in $M$ transversal to $K$ such that the $1$-cycle $\partial D - L$ is supported in $\partial M$ 
and represents an element of $\mathbf{E} \subset H_1(\partial M;\Q)$.
It is easily observed (see \cite[\S 1.3]{CT})  that
\begin{equation}\label{eq:sym}
\Lk_M^{\mathbf{E}}(L,K) - \Lk_M^{\mathbf{E}}(K,L) = e([L]) \mediumdot_{\!\partial M}\, e([K])
\end{equation}
where $\mediumdot_{\!\partial M}$ denotes the homology intersection in $\partial M$. 
In particular, the  invariant $\Lk_M^{\mathbf{E}}(-,-)$ is symmetric if   $\mathbf{E}$ is isotropic (and hence Lagrangian) for the symplectic form $ \mediumdot_{\!\partial M}$.

\subsection{Dependence on the essential subspace}

Let $M$ be a $\Q$-homology handlebody with Lagrangian  $\mathbf{L}:= \mathbf{L}_M^\Q \subset H_1(\partial M;\Q)$.
We now study how generalized linking numbers depend on the choice of the essential subspace in $H_1(\partial M;\Q)$.
First,  we emphasize the homological nature of linking numbers.

\begin{lemma} \label{lem:P&Q}
Let $P,Q\subset \operatorname{int}(M)$  be two submanifolds of dimension $3$ such that $P \cap Q =Ê\varnothing$,
and let ${\mathbf{E}}$ be an essential subspace of $H_1(\partial M;\Q)$.
Then, there is a unique bilinear map $\ell_{M}^{\mathbf{E}}:H_1(P;\Q) \times H_1(Q;\Q) \to \Q$ such that
$$
\ell_{M}^{\mathbf{E}}([K],[L])=\Lk_M^{\mathbf{E}}(K,L) 
$$ 
for any oriented knots $K\subset P$ and $L\subset Q$.
\end{lemma}

\begin{proof}
For $N=P$ or $Q$, denote by $\mathcal{K}_N$ the set of oriented knots in $N$ 
and observe that the map $\Q \times \mathcal{K}_N \to H_1(N;\Q)$ defined by $(p,K) \mapsto p[K]$ is surjective. Consider the map
$$
\tilde\ell_{M}^{\mathbf{E}}: (\Q \times \mathcal{K}_P) \times (\Q \times \mathcal{K}_Q) \longrightarrow \Q, \ \big((p,K),(q,L)\big) \longmapsto pq \Lk_M^{\mathbf{E}}(K,L).
$$
It follows from \eqref{eq:lk^S} that
$$
\tilde\ell_{M}^{\mathbf{E}}\big((p,K),(q,L)\big) \cdot [m_K] = p\cdot (q[L]) - p \cdot \incl_* e(q[L])   \in H_1(M\setminus K;\Q)
$$
which shows that, for fixed $(p,K)$, the number $\tilde\ell_{M}^{\mathbf{E}}((p,K),(q,L))$ only depends (linearly) on $q[L] \in H_1(Q;\Q)$ . Furthermore, we have
\begin{eqnarray*}
\tilde\ell_{M}^{\mathbf{E}}\big((p,K),(q,L)\big) & \stackrel{\eqref{eq:sym}}{=} & \tilde\ell_{M}^{\mathbf{E}}\big((q,L),(p,K)\big) + e(q[K])\mediumdot_{\!\partial M}\, e(p[L])
\end{eqnarray*}
so that, for a fixed $(q,L)$, $\tilde\ell_{M}^{\mathbf{E}}((p,K),(q,L))$  only depends (linearly) on $p[K] \in H_1(P;\Q)$. 
Therefore, $\tilde\ell_{M}^{\mathbf{E}}$ factorizes to a unique bilinear map $\ell_{M}^{\mathbf{E}}:H_1(P;\Q) \times H_1(Q;\Q) \to \Q$.
\end{proof}

Next,  we define a  bilinear pairing in $H_1(\partial M;\Q)$.

\begin{lemma} \label{lem:ell_boundary}
Let ${\mathbf{E}}$ be an essential subspace of $H_1(\partial M;\Q)$.
The generalized linking number $\Lk_M^{\mathbf{E}}(-,-)$ induces a  bilinear form $\vartheta^{\mathbf{E}}_M:H_1(\partial M;\Q) \times H_1(\partial M;\Q) \to \Q$.
Moreover, the left radical	and right radical of $ \vartheta_M^{\mathbf{E}}$ are given by 
$$
\left\{x\in H_1(\partial M;\Q): \vartheta_M^{\mathbf{E}}(x,-)=0 \right\}=\mathbf{L} \quad \hbox{and}  \quad
\left\{y\in H_1(\partial M;\Q): \vartheta_M^{\mathbf{E}}(-,y)=0 \right\}= \mathbf{E},
$$
respectively,  and $\vartheta_M^{\mathbf{E}}$ is the opposite of the  intersection form $\mediumdot_{\!\partial M}$ on  $\mathbf{E} \times \mathbf{L}$.
\end{lemma}

\begin{proof}
Consider a collar neighborhood $ \partial M \times [-1,0]$ of $\partial M \equiv \partial M \times \{0\}$. 
By applying Lemma \ref{lem:P&Q} to $P= \partial M \times ]-1,-1/2[$ and $Q = \partial M \times ]-1/2,0[$, 
we get a bilinear map $\vartheta^{\mathbf{E}}_M:H_1(\partial M;\Q) \times H_1(\partial M;\Q) \to \Q$. 

Set ${}^*\!R := \{x\in H_1(\partial M;\Q): \vartheta_M^{\mathbf{E}}(x,-)=0 \}$, $R^* := \{y\in H_1(\partial M;\Q): \vartheta_M^{\mathbf{E}}(-,y)=0 \}$.
We show that $\mathbf{E} \subset R^* $.
Let $y\in \mathbf{E}$ and let  $L \subset \partial M \times ]-1/2,0[$  be an oriented knot representing $y$ in  $H_1(\partial M \times ]-1/2,0[;\Q) \simeq H_1(\partial M;\Q)$. 
Then, for any oriented knot $K\subset \partial M \times ]-1,-1/2[$, we have
$$
\Lk_M^{\mathbf{E}}(K,L) \cdot [m_K] \stackrel{\eqref{eq:lk^S}}{=} [L]-  \incl_* e([L])   =[L] -[L] =0 \in H_1(M\setminus K;\Q)
$$
so that $\Lk_M^{\mathbf{E}}(K,L)=0$; this implies that $\vartheta_M^{\mathbf{E}}(-,y)=0$.

We show that $\vartheta_M^{\mathbf{E}}(\mathbf{L},\mathbf{L})= 0$ which, by the previous paragraph, implies that $\mathbf{L} \subset {}^*\!R$.
Let $x,y\in \mathbf{L} \subset H_1(\partial M;\Q) \simeq H_1( \partial M \times [-1,0];\Q)$ 
and  represent them by some oriented knots $K \subset \partial M \times]-1,-1/2[, L \subset \partial M \times ]-1/2,0[$, respectively. It follows from \eqref{eq:sym} that
$
\vartheta^{\mathbf{E}}_M(x,y)=  \Lk^{\mathbf{E}}_M(K,L)= \Lk^{\mathbf{E}}_M(L,K) . 
$
Since $[K]=0 \in H_1(M\setminus(\partial M \times ]-1/2,0]) ;\Q)$, we have $[K]=0 \in H_1(M\setminus L ;\Q)$ so that $ \Lk^{\mathbf{E}}_M(L,K)=0$.

We prove that $\vartheta_M^{\mathbf{E}}$ coincides with the opposite of $\mediumdot_{\!\partial M}$ on $\mathbf{E} \times \mathbf{L}$.
Let $x,y \in  H_1(\partial M;\Q) \simeq H_1( \partial M \times [-1,0];\Q)$ such that $x \in \mathbf{E}$ and $y\in \mathbf{L}$.
We represent $x,y$ by oriented knots $K \subset \partial M \times ]-1,-1/2[$ and $L \subset \partial M \times ]-1/2,0[$,  respectively.
Choose a rational $2$-chain $D$ transversal to $K$ such that $\partial D=L$.
Then
$$
\vartheta_M^{\mathbf{E}}(x,y) = \Lk_M^{\mathbf{E}}(K,L) \stackrel{\eqref{eq:lk_as_intersection}}{=} K \mediumdot_{\!M}\, D = -x \mediumdot_{\!\partial M}\, y.
$$

We now prove that ${}^*\!R \subset \mathbf{L}$ and $R^* \subset \mathbf{E}$.
Denote by $ p_{\mathbf{E}}: H_1(\partial M;\Q) \to \mathbf{L}$ and $ q_{\mathbf{E}}: H_1(\partial M;\Q) \to \mathbf{E}$ 
the projections of the direct sum $H_1(\partial M;\Q) = \mathbf{L} \oplus \mathbf{E}$. 
Let $x\in H_1(\partial M;\Q)$ such that $\vartheta_M^{\mathbf{E}}(x,-)=0$; then, for all $l\in \mathbf{L}$, 
$$
x \mediumdot_{\!\partial M}\, l= p_{\mathbf{E}}(x) \mediumdot_{\!\partial M}\, l + q_{\mathbf{E}}(x) \mediumdot_{\!\partial M}\, l
= -\vartheta_M^{\mathbf{E}}(q_{\mathbf{E}}(x) ,l) = -\vartheta_M^{\mathbf{E}}(x,l)=0;
$$
since $\mathbf{L}$ is Lagrangian with respect to $\mediumdot_{\!\partial M}$, it follows that $x \in \mathbf{L}$. We conclude that ${}^*\!R = \mathbf{L}$.
Let $y\in H_1(\partial M;\Q)$ such that $\vartheta_M^{\mathbf{E}}(-,y)=0$; then, for all $x\in H_1(\partial M;\Q)$,
\begin{eqnarray*}
x\mediumdot_{\!\partial M}\, p_{\mathbf{E}}(y) & = & p_{\mathbf{E}}(x) \mediumdot_{\!\partial M}\, p_{\mathbf{E}}(y)  + q_{\mathbf{E}}(x) \mediumdot_{\!\partial M}\, p_{\mathbf{E}}(y)   \\
&=&   - \vartheta_M^{\mathbf{E}}( q_{\mathbf{E}}(x) ,p_{\mathbf{E}}(y)) \ = \ -\vartheta_M^{\mathbf{E}}( q_{\mathbf{E}}(x) , y) \ = \ 0;
\end{eqnarray*}
we deduce that $p_{\mathbf{E}}(y) =0$, i.e$.$ $y \in \mathbf{E}$. We conclude that  $R^* = \mathbf{E}$.
\end{proof}

The form $\vartheta_M^{\mathbf{E}}$ is the ``generalized Seifert form'' of the surface $-\partial M$ as defined in \cite[\S 2]{CT}.
It is not symmetric, but it satisfies
$$
\forall x,y \in H_1(\partial M;\Q), \quad \vartheta_M^{\mathbf{E}}(y,x) - \vartheta_M^{\mathbf{E}}(x,y) 
= e(\incl_*(y) ) \mediumdot_{\!\partial M}\, e( \incl_*(x) ) -  y \mediumdot_{\!\partial M}\, x .
$$ 
(This can be deduced from identity \eqref{eq:sym}, see also \cite[\S 2.1]{CT}.) 
This form measures how  the generalized linking number depends on the choice of the essential subspace, as the next lemma shows.

\begin{lemma}\label{lem:Lk_ell}
Let ${\mathbf{E}},{\mathbf{F}}$ be essential subspaces of $H_1(\partial M;\Q)$. For any disjoint, oriented knots $K,L\subset \operatorname{int}(M)$, we have
\begin{equation}\label{eq:variation_Lk}
\Lk_M^{\mathbf{F}}(K,L) = \Lk_M^{\mathbf{E}}(K,L) -  \vartheta_M^{\mathbf{E}}\big(f([K]),f([L])\big)
\end{equation}
where $f: H_1(M;\Q) \to H_1(\partial M;\Q)$ is the section of $\incl_*: H_1(\partial M;\Q) \to H_1(M;\Q)$ corresponding to $\mathbf{F}$.
\end{lemma}

\begin{proof}
Consider a collar neighborhood $\partial M \times [-1,0]$ of $\partial M \equiv \partial M \times \{0\}$, which does not cut $K \cup L$.
Let $\tilde K \subset \partial M \times ]-1,-1/2[$ and $\tilde L \subset \partial M \times ]-1/2,0[$ be oriented knots  such that  
\begin{eqnarray*}
[\tilde K ]&=& f([K]) \in H_1(\partial M\times ]-1,-1/2[;\Q) \simeq H_1(\partial M;\Q)\\ 
\quad \hbox{and} \quad [\tilde L ]&=& f([L]) \in H_1(\partial M\times ]-1/2,0[;\Q) \simeq H_1(\partial M;\Q).
\end{eqnarray*}
We have to prove that
\begin{equation} \label{eq:3xLk}
\Lk_M^{\mathbf{E}}(\tilde K,\tilde L) = \Lk_M^{\mathbf{E}}(K,L) -  \Lk_M^{\mathbf{F}}(K,L).
\end{equation}
Since $\tilde K$ is rationally homologous to $K$ in the exterior of  $\partial M \times ]-1/2,0]$ and since $\tilde L$ is contained in  $\partial M \times ]-1/2,0]$,
$\tilde K$ is rationally homologous to $K$ in $M \setminus \tilde L$. Therefore $ \Lk_M^{\mathbf{E}}(\tilde K,\tilde L) =  \Lk_M^{\mathbf{E}}(K,\tilde L)$ and we obtain
\begin{eqnarray*}
\Lk_M^{\mathbf{E}}(\tilde K,\tilde L) \cdot [m_K] &=&  [\tilde L] -\incl_* e([\tilde L]) \\
&=& \incl_* f([L]) -\incl_* e([ L])  \ \in H_1(M\setminus K;\Q).
\end{eqnarray*}
Besides, we have
\begin{eqnarray*}
\Lk_M^{\mathbf{E}}(K,L) \cdot [m_K] - \Lk_M^{\mathbf{F}}(K,L) \cdot [m_K] &=& \big( [L] - \incl_* e([L])\big) -   \big( [L] - \incl_* f([L])\big) \\
&=& \incl_* f([L]) -\incl_* e([ L])  \ \in H_1(M\setminus K;\Q).
\end{eqnarray*}
Identity \eqref{eq:3xLk} follows.
\end{proof}

\subsection{Essential Jacobi diagrams}   \label{subsec:eJd} 

Let $M$ be a $\Q$-homology handlebody.
For any essential subspace ${\mathbf{E}}$ of $H_1(\partial M;\Q)$, we  consider the subspace 
$
\A^Y({\mathbf{E}}) \subset \A({\mathbf{E}})
$
spanned by Jacobi diagrams without strut component. If  ${\mathbf{F}} \subset H_1(\partial M;\Q)$ is another essential subspace,
we would like to identify  $\A^Y({\mathbf{E}}) $ and $\A^Y({\mathbf{F}})$ in a canonical way. 
Of course, there is the obvious isomorphism $\rho_M^{{\mathbf{F}},{\mathbf{E}}}:  \A^Y({\mathbf{F}}) \to \A^Y({\mathbf{E}})$ that consists in changing the colors of univalent vertices 
by means of the isomorphism $ef^{-1}:{\mathbf{F}} \stackrel{\simeq}{\longrightarrow} H_1(M;\Q)  \stackrel{\simeq}{\longrightarrow}  {\mathbf{E}}$.
But,  this is not enough for our purposes.

To go further, we assume  that ${\mathbf{E}}$ and ${\mathbf{F}}$ are also isotropic subspaces of $H_1(\partial M;\Q)$.
Then the restriction of the form  $\vartheta_M^{\mathbf E}$ to $\mathbf{F} \times \mathbf{F}$ is symmetric by Lemma \ref{lem:Lk_ell}.
So we can consider the linear map
$$
\kappa_M^{{\mathbf{F}},{\mathbf{E}}}: \A^Y({\mathbf{F}}) \longrightarrow \A^Y({\mathbf{E}})
$$
defined for any ${\mathbf{F}}$-colored Jacobi diagram $D$ by 
$$
\kappa_M^{{\mathbf{F}},{\mathbf{E}}}(D):= 
\rho_M^{{\mathbf{F}},{\mathbf{E}}}\left( \hbox{sum of all ways of gluing some legs of $D$ with  $\vartheta_M^{\mathbf{E}}$}\right).
$$
Observe that $\kappa_M^{{\mathbf{E}},{\mathbf{E}}}$ is the identity of $\A^Y({\mathbf{E}})$ since $\vartheta_M^{\mathbf{E}}(\mathbf{E},\mathbf{E})=0$ by Lemma \ref{lem:ell_boundary}. 
It follows from the next lemma that $\kappa_M^{{\mathbf{F}},{\mathbf{E}}}$ is an isomorphism for any ${\mathbf{E}},{\mathbf{F}}$.

\begin{lemma}
Let ${\mathbf{E}},{\mathbf{F}},{\mathbf{G}}$ be essential isotropic subspaces of $H_1(\partial M;\Q)$. 
Then  we have $\kappa_M^{{\mathbf{F}},{\mathbf{E}}} \circ \kappa_M^{{\mathbf{G}},{\mathbf{F}}} = \kappa_M^{{\mathbf{G}},{\mathbf{E}}}$.
\end{lemma}

\begin{proof}
We need the following identity which is a direct consequence of   Lemma \ref{lem:Lk_ell}: 
\begin{equation}\label{eq:transitivity}
\forall x,y \in {\mathbf{G}}, \qquad \vartheta_M^{\mathbf{E}}(x,y)=\vartheta_M^{\mathbf{F}}(x,y)+ \vartheta_M^{\mathbf{E}}\big(fg^{-1 }(x),fg^{-1 }(y)\big).
\end{equation}
Let $D$ be a ${\mathbf{G}}$-colored Jacobi diagram and consider $k$ pairs $\{v_1,w_1\},\dots, \{v_k,w_k\}$ of distinct univalent vertices of $D$
such that $\{v_i,w_i\}\cap \{v_j,w_j\} = \varnothing$ for any $i\neq j$. 
By making the identification $v_i\equiv w_i$ for every $i\in \{1,\dots,k\}$ and by applying the isomorphism $eg^{-1}$ to the colors of the remaining univalent vertices,
we obtain an $\mathbf{E}$-colored Jacobi diagram $D'$. 
The coefficient of $D'$ in the definition of $\kappa_M^{{\mathbf{G}},{\mathbf{E}}}(D)$ is 
\begin{eqnarray*}
&&\prod_{j=1}^k  \vartheta_M^{\mathbf{E}}\left(\col(v_j),\col(w_j)\right) \\
&\stackrel{\eqref{eq:transitivity}}{=} &\prod_{j=1}^k  \left( \vartheta_M^{\mathbf{F}}\big(\col(v_j),\col(w_j)\big)+ \vartheta_M^{\mathbf{E}}\big(fg^{-1 }\col(v_j),fg^{-1 }\col(w_j)\big)\right) \\
&=& \sum_{P\subset \{1,\dots,k\}}\  \prod_{p\in P}  \vartheta_M^{\mathbf{F}}\big(\col(v_p),\col(w_p)\big) \cdot  \prod_{q\not\in P}  \vartheta_M^{\mathbf{E}}\big(fg^{-1}\col(v_q),fg^{-1}\col(w_q)\big),
\end{eqnarray*}
which is also the coefficient of $D'$ in $\kappa_M^{{\mathbf{F}},{\mathbf{E}}}\big(\kappa_M^{{\mathbf{G}},{\mathbf{F}}}(D)\big)$.
\end{proof}

\subsection{$\Q$-LP surgery equivalence} \label{subsec:Q-LP}

We now show that the previous  constructions relative to a $\Q$-homology handlebody $M$
only depend on the $\Q$-LP surgery equivalence class of $M$. 

\begin{lemma} \label{lem:pre_Q-LP}
Let $M$  be a $\Q$-homology handlebody, let $\mathsf{C}=(C',C'')$ be  a $\Q$-LP pair such that $C' \subset \operatorname{int}(M)$
and let $M_{\mathsf{C}}$ be the result of the $\Q$-LP  surgery.  
Then there is a unique isomorphism $\psi:H_1(M;\Q) \to H_1(M_{\mathsf{C}};\Q)$  such that the following diagram is commutative:
$$
\xymatrix{
H_1(\partial M;\Q) \ar[r]^-{\incl_*} \ar@{=}[d]& H_1(M;\Q) \ar[d]^-{\psi}_-\simeq\\
H_1(\partial M_{\mathsf{C}};\Q) \ar[r]_-{\incl_*} & H_1(M_{\mathsf{C}};\Q)
}
$$
Moreover, any subspace $\mathbf{E} \subset H_1(\partial M;\Q)$ essential for $M$ is also essential for $M_{\mathsf{C}}$  and 
$
\Lk_M^{\mathbf{E}}(K,L) = \Lk_{M_{\mathsf{C}}}^{\mathbf{E}}(K,L)
$
for any disjoint, oriented knots $K,L \subset M\setminus \operatorname{int}(C')$.
\end{lemma}

\begin{proof}
The unicity of $\psi$ follows from the surjectivity of $\incl_*:H_1(\partial M;\Q)\to H_1(M;\Q)$. 
The existence is an application of the Mayer--Vietoris theorem showing  that there is a unique isomorphism $\psi:H_1(M;\Q) \to H_1(M_{\mathsf{C}};\Q)$ 
making the following diagram commutative:
$$
\xymatrix @!0 @R=1cm @C=3cm {
& H_1(M;\Q) \ar@{-->}[dd]_-\simeq^-\psi \\
H_1\big(M\setminus \operatorname{int}(C');\Q\big)  \ar@{->>}[ru]^-{\incl_*}  \ar@{->>}[rd]_-{\incl_*} & \\
& H_1(M_{\mathsf{C}};\Q)
}
$$
Assume that $\mathbf{E} \subset H_1(\partial M;\Q)$ is essential for $M$. Since $\psi$ is an isomorphism, 
$\incl_* \vert_{\mathbf{E}}= \psi \circ \incl_*  \vert_{\mathbf{E}}: \mathbf{E} \to H_1(M_{\mathsf{C}};\Q)$ is an isomorphism so that $\mathbf{E}$ is essential for $M_{\mathsf{C}}$ as well.
Let $K,L \subset M\setminus \operatorname{int}(C')$ be any two disjoint oriented knots.
Let $D$ be a rational $2$-chain in $M$ transversal  to $K$ such that $\tilde L :=\partial D -L$ is a $1$-cycle in $\partial M$ and  $\big[\tilde L \big] \in H_1(\partial M;\Q)$ belongs to $\mathbf{E}$.
We can assume that $D=D_0+ D'$ where $D_0$ is a rational $2$-chain in $M\setminus \operatorname{int}(C')$, $D'$ is a rational $2$-chain in $C'$ and $\partial D'$ is a rational $1$-cycle in $\partial C'$.  
Since $\mathbf{L}_{C'}^\Q=\mathbf{L}_{C''}^\Q$, there is a rational $2$-chain $D''$ in $C''$ such that $\partial D'' = \partial D'$. 
Then $D_{\mathsf{C}}:=D_0+ D''$ is a rational $2$-chain in $M_{\mathsf{C}}$ with boundary $\tilde L + L$ and which is transversal to $K$.
We conclude thanks to  \eqref{eq:lk_as_intersection}:
$$
\Lk_M^{\mathbf{E}}(K,L) = K \mediumdot_{\! M}\, D =  K \mediumdot_{\!M_{\mathsf{C}}}\, D_{\mathsf{C}} =  \Lk_{M_{\mathsf{C}}}^{\mathbf{E}}(K,L)
$$

\up
\end{proof}

We now fix a closed connected oriented surface $F$.
Given a $\Q$-homology handlebody $M$ with boundary parameterization $m:F \to M$,
a subspace $\mathbf{E}$ of $H_1(F;\Q)$ is said to be \emph{$M$-essential} if $m_*(\mathbf{E})$ is essential in the sense of \S \ref{subsec:def_lk},
i.e$.$ the restriction to $\mathbf{E}$  of $m_*:H_1(F;\Q) \to H_1(M;\Q)$ is an isomorphism onto $H_1(M;\Q)$.

\begin{lemma} \label{lem:Q-LP}
Let $(M,m)$ and $(\bar M, \bar m)$ be $\Q$-homology handlebodies whose boundaries are parameterized  by $F$. Then the following statements are equivalent:
\begin{enumerate}
\item[(i)] $(\bar M, \bar m)$ can be obtained from $(M,m)$ by  a single $\Q$-LP surgery;
\item[(i')] $(\bar M, \bar m)$ can be obtained from $(M,m)$ by  a sequence of $\Q$-LP surgeries;
\item[(ii)] we have $m_*^{-1} \big(\mathbf{L}_M^\Q \big) = \bar m_*^{-1} \big( \mathbf{L}_{\bar M}^\Q \big)$ in $H_1(F;\Q)$;
\item[(iii)] there is  a subspace $\mathbf{E} \subset H_1(F;\Q)$ which is  $M$-essential and $\bar M$-essential, such that
 $\vartheta_M^{m_*(\mathbf{E})} \circ (m_*\times m_*)  = \vartheta_{\bar M}^{\bar m_*(\mathbf{E})} \circ (\bar m_*\times \bar m_*)$.
\end{enumerate}
\end{lemma}

\begin{proof}
Clearly  (i) implies (i'). Assume (i'): there is a sequence of $\Q$-LP surgeries 
$$
(M,m)=(M_1,m_1) \leadsto (M_2,m_2)  \leadsto \cdots \leadsto (M_{r+1},m_{r+1}) =(\bar M, \bar m).
$$ 
For all $i\in \{1,\dots,r\}$, the isomorphism $\psi_i:H_1(M_i;\Q) \to H_1(M_{i+1};\Q)$ given by Lemma \ref{lem:pre_Q-LP} satisfies $ \psi_i \circ m_{i,*}  = m_{i+1,*}:H_1(F;\Q) \to H_1(M_{i+1};\Q)$.
Hence the composite isomorphism $\psi := \psi_r \circ \cdots \circ \psi_{1}$ satisfies   $\psi \circ m_{*}= \bar m_{*}$.
For any $x\in m_*^{-1} (\mathbf{L}_M^\Q)$, we have $$\incl_* \bar m_*(x) =\bar m_*(x)=  \psi m_*(x)  =   \psi\, \incl_* m_*(x) =0$$
so that $x\in \bar m_*^{-1} (\mathbf{L}_{\bar M}^\Q)$. This shows that  $ m_*^{-1} (\mathbf{L}_M^\Q) \subset \bar m_*^{-1} (\mathbf{L}_{\bar M}^\Q)$,
and the opposite inclusion is proved similarly: hence we get (ii).
Clearly, (ii)  implies (i) and it now remains to prove the equivalence between (iii) and (i),(i'),(ii).

That (i) implies (iii) follows from the second statement of Lemma \ref{lem:pre_Q-LP}.
Assume  (iii) and denote by $h:\partial M \to \partial \bar M$ the orientation-preserving homeomorphism defined by $h(m(x))=\bar m(x) $ for any $x\in F$.
Then we have  $\vartheta_M^{m_*(\mathbf{E})}  = \vartheta_{\bar M}^{\bar m_*(\mathbf{E})} \circ (h_*\times h_*)$ so that 
$h_*$ sends the left radical of  $\vartheta_M^{m_*(\mathbf{E})}$ to the left radical of $\vartheta_{\bar M}^{\bar m_*(\mathbf{E})}$. 
We deduce assertion (ii) using the second statement of Lemma \ref{lem:ell_boundary}.
\end{proof}

Fix a $\Q$-LP surgery equivalence class $\mathcal{M}$ of $\Q$-homology handlebodies with boundary parameterized by $F$. 
By the equivalence (i)$\Leftrightarrow$(ii) in Lemma \ref{lem:Q-LP},
the choice of $\mathcal{M}$ is equivalent to the choice of a Lagrangian subspace $\mathbf{L}$ of $H_1(F;\Q)$, satisfying $m_*(\mathbf{L}) = \mathbf{L}_{M}^\Q$ for any $M\in \mathcal{M}$.
A subspace $\mathbf{E}$ of  $H_1(F;\Q)$ is said to be \emph{$\mathcal{M}$-essential} if $H_1(F;\Q) = \mathbf{L} \oplus \mathbf{E}$ and,
in this case, it induces a  bilinear form
$$
\vartheta_{\mathcal{M}}^{\mathbf{E}}: H_1(F;\Q) \times H_1(F;\Q) \longrightarrow \Q
$$
defined by $\vartheta_{\mathcal{M}}^{\mathbf{E}}(l_1+e_1,l_2+e_2):=  l_2 \mediumdot_{\!F}\, e_1$ for any $l_1,l_2 \in \mathbf{L}$ and $e_1,e_2 \in \mathbf{E}$.
Observe that $\mathbf{E}$ is $\mathcal{M}$-essential if and only if it is $M$-essential for a particular $M\in \mathcal{M}$
and, by Lemma  \ref{lem:ell_boundary},   $\vartheta_{\mathcal{M}}^{\mathbf{E}}$ then corresponds  to  $\vartheta_{M}^{m_*(\mathbf{E})}$ 
through the isomorphism $m_*:H_1(F;\Q) \to H_1(\partial M;\Q)$.

Furthermore, for any two isotropic $\mathcal{M}$-essential subspaces  $\mathbf{E}, \mathbf{F} \subset H_1(F;\Q)$, there is a linear map
$$
\kappa_{\mathcal{M}}^{\mathbf{F}, \mathbf{E}}: \A^Y(\mathbf{F}) \longrightarrow \A^Y(\mathbf{E})
$$
defined for any ${\mathbf{F}}$-colored Jacobi diagram $D$ by 
$$
\kappa_{\mathcal{M}}^{{\mathbf{F}},{\mathbf{E}}}(D) := 
\rho_{\mathcal{M}}^{{\mathbf{F}},{\mathbf{E}}}\left( \hbox{sum of all ways of gluing {some} legs of $D$ with  $\vartheta_{\mathcal{M}}^{\mathbf{E}}$}\right)
$$
where $\rho_{\mathcal{M}}^{{\mathbf{F}},{\mathbf{E}}}: \A^Y(\mathbf{F}) \to \A^Y(\mathbf{E})$ is the isomorphism induced by the change of colors 
$$
\xymatrix{
\mathbf{F} \ar[rr]^-{\hbox{\scriptsize inclusion}} &&\mathbf{L} \oplus \mathbf{F} = H_1(F;\Q) = \mathbf{L} \oplus \mathbf{E} \ar[rr]^-{\hbox{\scriptsize projection}} && \mathbf{E}.
}
$$
Observe that the map $\kappa_{\mathcal{M}}^{{\mathbf{F}},{\mathbf{E}}}$ corresponds to the isomorphism $\kappa_{M}^{m_*(\mathbf{F}),m_*(\mathbf{E})}$
for any $M\in \mathcal{M}$ through the isomorphisms $\A^Y(\mathbf{E}) \simeq \A^Y(m_*(\mathbf{E}))$ and $\A^Y(\mathbf{F}) \simeq \A^Y(m_*(\mathbf{F}))$
induced by the changes of colors $m_*\vert_{\mathbf{E}}: \mathbf{E}  \to m_*(\mathbf{E}) \subset  H_1(\partial M;\Q)$ and 
$m_*\vert_{\mathbf{F}}: \mathbf{F}  \to m_*(\mathbf{F}) \subset  H_1(\partial M;\Q)$, respectively.

\section{Milnor's triple linking numbers in $\Q$-homology $3$-spheres}   \label{sec:Milnor}

We extend the definition of Milnor's triple linking numbers in $S^3$ to any $\Q$-homology $3$-sphere
and we relate them to the Kontsevich--LMO invariant. This relation is needed in the proof of the splitting formulas. 

\subsection{Definition} \label{subsec:def_mu123}

Let $L$ be an algebraically-split oriented link in a $\Q$-homology $3$-sphere $S$,
whose connected components are numbered from $1$ to $\ell$:
$$
\forall i\neq j \in \{1,\dots,\ell\}, \ \Lk_S(L_i,L_j) = 0.
$$
Then, for each triplet $(i,j,k)$ of distinct integers in $\{1,\dots,\ell\}$,
there is an isotopy invariant of the link $L$ in $S$
$$
\bar\mu_{ijk}(L) \in \Q
$$
which, in the case of $S=S^3$, is the usual Milnor's isotopy invariant of length $3$.
In general, $\bar\mu_{ijk}(L)$ is not an integer but a rational number.

The most efficient, although indirect, way to define $\bar\mu_{ijk}(L)$ is as follows.
Let $S_L$ be the closed oriented $3$-manifold obtained from $S$ by longitudinal surgery along $L$.
Thus, $S_L$ is obtained from the exterior $S\setminus \operatorname{int}(N(L))$ of the link $L$
by gluing $\ell$ solid tori, the meridian of the $i$-th solid torus being glued to the longitude\footnote{ 
The \emph{longitude} of an oriented knot $K$ in a $\Q$-homology $3$-sphere $S$
is the unique oriented simple closed curve on  $\partial N(K)$ that is rationally null-homologous in 
$S\setminus \operatorname{int}(N(K))$ and that is homotopic in $N(K)$ to $b\cdot K$ for some integer $b>0$.
This is not necessarily a parallel of $K$, i.e we may have $b>1$.}
$\lambda_i$ of the $i$-th component of $L$.
The vector space $H_1(S_L;\Q)$ has dimension $\ell$, 
with preferred basis given by the images of the oriented meridians $m_1,\dots,m_\ell$ of $L_1,\dots, L_\ell$ under 
$\incl_*: H_1(S\setminus \operatorname{int}(N(L));\Q) \to H_1(S_L;\Q)$.
Let also $(m^*_1,\dots, m_\ell^*)$ be the dual basis of $H^1(S_L;\Q)$, i.e$.$
$\langle m^*_i, m_j \rangle = \delta_{i,j}$ for all $i,j=1,\dots, \ell$. Then,
\begin{equation}\label{eq:mu}
\bar \mu_{ijk}(L) := \big\langle m_i^* \cup m_j^* \cup m_k^* , [S_L] \big\rangle \ \in \Q
\end{equation}
defines an isotopy invariant of $L \subset S$.

The invariant $\bar\mu_{ijk}(L)$ can be computed as follows.
Since the link $L$ is assumed to be algebraically-split, each component $L_i$ of $L$
is rationally null-homologous in the exterior of the other components.
Thus, we can find some compact connected oriented surfaces 
$\Sigma_1, \dots, \Sigma_\ell \subset S$ satisfying the following:
\begin{itemize}
\item for all $i\in \{1,\dots,\ell \}$, there exists an integer $n_i>0$ such that 
$\partial \Sigma_i$ winds  $n_i$ times around $L_i$;
\item for all $i\neq j \in \{1,\dots,\ell \}$, $\Sigma_i$ is in transverse position with $\Sigma_j$ 
and $\Sigma_i \cap L_j = \varnothing$.
\end{itemize}

\begin{lemma}\label{lem:triple_Milnor}
With the above notation, we have
\begin{equation}
\label{eq:Milnor}
\bar \mu_{ijk}(L) = - \frac{\Sigma_i \mediumdot \Sigma_j \mediumdot \Sigma_k}{n_i n_j  n_k} \in \Q.
\end{equation}
\end{lemma}

\noindent
In this form, the invariant $\bar \mu_{ijk}(L)$ appears in \cite{Lescop_KKT_surgery}.
Here, the intersection number $\Sigma_i \mediumdot \Sigma_j \mediumdot \Sigma_k$ in $S$ is computed
with the sign convention of \cite{Bredon} (which agrees with that of \cite{Lescop_KKT_surgery}).

%
%

\begin{proof}[Proof of Lemma \ref{lem:triple_Milnor}]
We can assume after an isotopy that, for each $i\in \{1,\dots,\ell \}$,
the surface $\Sigma_i$ does not cut $N(L_j)$ for $j\neq i$
and is in tranverse position with  $\partial N(L_i)$.
Thus, $\Sigma_i \cap \partial  N(L_i)$ consists of null-homotopic simple closed curves in $\partial N(L_i)$
and a family of simple closed curves parallel to the same curve. 
This curve being then rationally null-homologous in the complement of $N(L_i)$,
it is necessarily the longitude $\lambda_i$ of $L_i$: let $r_i\in \Z$ be the number of times
(counted with signs) $\Sigma_i \cap (S\setminus \operatorname{int}(N(L)))$ winds around $\lambda_i$.
Because the surgery $S \leadsto S_L$ is longitudinal, we can transform 
$\Sigma_i \cap (S\setminus \operatorname{int}(N(L)))$ to a closed oriented surface $(\Sigma_i)_L\subset S_L$ 
by gluing some disks  inside the $i$-th surgery solid torus.
Let $p_i\subset \partial N(L_i)$ be an oriented parallel of $L_i$,
and let $a_i,b_i$ be the integers satisfying $\lambda_i =a_i \cdot m_i + b_i \cdot p_i \in H_1(\partial N(L_i);\Z)$.
We have $(\Sigma_i)_L \cap m_j=\varnothing $ for all $j\neq i$ and 
$$
m_i \mediumdot (\Sigma_i)_L = m_i\mediumdot (r_i \lambda_i) = r_i b_i = n_i.
$$
Therefore,  $[\frac{1}{n_i}(\Sigma_i)_L]\in H_2(S_L;\Q)$ is Poincar\'e dual to $m_i^* \in H^1(S_L;\Q)$.
We conclude thanks to the Poincar\'e correspondence between homological intersections and cup products 
(using the conventions of \cite{Bredon}):
\begin{eqnarray*}
\bar\mu_{ijk}(L) & = & \big\langle m_i^* \cup m_j^* \cup m_k^* , [S_L] \big\rangle  \\
&=&  -\left[\frac{1}{n_i}(\Sigma_i)_L\right]\mediumdot\left[\frac{1}{n_j}(\Sigma_j)_L\right]
 \mediumdot \left[\frac{1}{n_k}(\Sigma_k)_L\right] 
 \ = \ - \frac{\Sigma_i \mediumdot \Sigma_j \mediumdot \Sigma_k}{n_i n_j  n_k}.
\end{eqnarray*}

\up
\end{proof}

\subsection{The ``Y'' part of the Kontsevich--LMO invariant}

The relationship between $\bar\mu$-invariants of links in $S^3$ and the Kontsevich integral
has been studied by Habegger \& Masbaum who considered   $\mu$-invariants of string-links in the standard cube \cite{HabeggerMasbaum}.
(See also \cite{Moffatt} for the case of string-links in $\Z$-homology  cubes.)
The following deals with the length three $\bar \mu$-invariants 
of algebraically-split links in $\Q$-homology $3$-spheres, as defined in \S \ref{subsec:def_mu123}.
 
\begin{lemma}
\label{lem:triple_Milnor_Y}
Let $(B,\gamma)$ be an $\ell$-component top tangle in a $\Q$-homology cube,
and assume that the plat closure $\hat \gamma$ of $\gamma$ 
in the $\Q$-homology $3$-sphere $\hat B$ is algebraically-split.
Then, the Kontsevich--LMO invariant of $(B,\gamma)$ 
determines the framing numbers\footnote{
The \emph{framing number} $\operatorname{Fr}(K)\in \Q$ of a framed knot $K$ in a $\Q$-homology $3$-sphere $S$ is
the linking number in $S$ of $K$ with the parallel of $K$ defined by the framing.} 
of $\hat \gamma$ and its Milnor's triple linking numbers as follows:
$$
\log_\sqcup \chi^{-1} Z(B,\gamma) = \sum_{i=1}^\ell \frac{\operatorname{Fr}(\hat \gamma_i)}{2} \strutgraphtop{i^+}{i^+}
 -\sum_{\substack 1 \leq i<j<k \leq \ell} \bar \mu_{ijk}(\hat{\gamma}) \cdot
\Ygraphtop{k^+}{j^+}{i^+} + (\ideg >1). 
$$
\end{lemma}

\noindent
Here the Kontsevich--LMO invariant $Z(B,\gamma)$ of the top tangle $(B,\gamma)$ 
is as defined in \cite[\S 3.5]{CHM} (and as briefly recalled in  \S \ref{sec:LMO_functor}).
The top tangle $\gamma$ is equipped with any non-associative word of length $\ell$
in the single letter $\bullet$ to which the rule  $\bullet \mapsto (+-)$ is applied.

\begin{proof}[Proof of Lemma \ref{lem:triple_Milnor_Y}]
The fact that $\chi^{-1}Z(B,\gamma) \in \A(\set{\ell}^+)$ is  of the form
$$
\exp_\sqcup\left(\sum_{i=1}^\ell \frac{\operatorname{Fr}(\hat \gamma_i)}{2} \strutgraphtop{i^+}{i^+}
+ \sum_{\substack 1 \leq i<j<k \leq \ell} a_{ijk}(\gamma) \cdot \Ygraphtop{k^+}{j^+}{i^+} + (\ideg >1)\right)
$$
for some $a_{ijk}(\gamma)\in \Q$ is well-known (see \cite[Lemma 3.17]{CHM} for example).
In the sequel, we denote by $z(B,\gamma)$ the ``Y'' part of $\chi^{-1} Z(B,\gamma)$
and by $m(B,\gamma)$ the linear combination of Y-shaped diagrams
defined by Milnor's triple linking numbers of $\hat{\gamma}$ in $\hat{B}$:
$$
z(B,\gamma) := \sum_{\substack 1 \leq i<j<k \leq \ell} a_{ijk}(\gamma) \cdot \Ygraphtop{k^+}{j^+}{i^+}, 
\quad \quad
m(B,\gamma) := \sum_{\substack 1 \leq i<j<k \leq \ell} \bar \mu_{ijk}(\hat{\gamma}) \cdot \Ygraphtop{k^+}{j^+}{i^+}.
$$
We are asked to prove that $m(B,\gamma)=-z(B,\gamma)$.

For any integer $r\geq 1$ and for any $i\in \{1,\dots,\ell\}$, 
let $\gamma_{r\times  i}$ be a top tangle in $B$ which is identical to $\gamma$,
except that the $i$-th component of $\gamma_{r\times i}$ now goes $r$ times around the 
$i$-th component of $\gamma$. Then, it can be deduced from the ``doubling property'' of the Kontsevich integral
that $z(B,\gamma_{r \times i})$ is obtained from $z(B,\gamma)$ by the operation $i^+ \mapsto r \cdot i^+$.
(Note that the version of the Kontsevich integral used in   \cite{CHM} has a ``doubling anomaly'', 
but this does not affect  $z(B,\gamma)$.)
Besides, it follows from Lemma \ref{lem:triple_Milnor} 
that $m(B,\gamma_{r\times  i})$ differs from $m(B,\gamma)$ in the same way.
Therefore we can assume in the sequel that each component of $\hat \gamma$ is null-homologous in $\hat B$ (with coefficients in $\Z$). 

For any  $n\in \Z$ and for any $i\in \{1,\dots,\ell\}$,
let $\gamma_{n \looparrowright  i}$ be the top tangle in $B$ that is obtained from $\gamma$
by changing the framing of $\gamma_i$ by $n$ 
(i.e$.$ by adding  $\vert n\vert$ ``kinks'' of  the same sign as $n$ to $\gamma$).
This operation modifies the Kontsevich integral by the exponential of an isolated chord (times $n/2$)
so that we have $z(B,\gamma)=z(B,\gamma_{n \looparrowright  i})$.
Since Milnor's triple linking numbers of $\hat \gamma$ do not depend on the framing of $\hat \gamma$,
we also have  $m(B,\gamma)=m(B,\gamma_{n \looparrowright  i})$.
So we can assume in the sequel that the framing number of each component of $\hat \gamma$ is zero.

With the above two assumptions on $\hat \gamma$, we can find for each $i\in \{1,\dots,\ell\}$ a Seifert surface $\Sigma_i$
for the framed knot $\hat \gamma_i$ such that $\Sigma_i$ is in transverse position with $\Sigma_j$ 
and does not meet $\hat\gamma_j$ for all $j\neq i$. 
The standard cube $[-1,1]^3$ can be obtained from $B$ by  surgery along a framed link $L^*$
and, by an isotopy of the link  $L^*$ in $B$, we can require that $L^* \subset \hat B$ is disjoint from $\hat \gamma$ and 
that each component of $L^*$  has a trivial linking number with every component of $\hat \gamma$.
Thus, by adding ``tubes'' to the Seifert surfaces  $\Sigma_1, \dots,  \Sigma_\ell$,
we can assume that each of them is disjoint from $L^*$.
Then the framed link $L\subset [-1,1]^3$ dual to $L^*$ has the following two properties:
first, surgery along $L$ produces $B$; second, the surfaces $\Sigma_1,\dots, \Sigma_\ell$ 
(and, a fortiori, their boundaries $\hat \gamma_1,\dots,\hat \gamma_\ell$)
can be seen in the exterior $[-1,1]^3 \setminus \operatorname{int}(N(L))$. 
Therefore, the pair $(L,\gamma)$ is a surgery presentation of $(B,\gamma)$ in $[-1,1]^3$
satisfying $\Lk_{[-1,1]^3}(L,\gamma)=0$ and $\Lk_{[-1,1]^3}(\gamma)=0$.
Then, using the formal Gaussian integration of \cite{BGRT1,BGRT2} 
and adopting the notation of \cite[\S 3]{CHM}, we have
\begin{eqnarray*}
z(B,\gamma) &=& \hbox{``Y'' part of} \quad U_+^{-\sigma_+(L)} \sqcup U_-^{-\sigma_-(L)} \sqcup
\int_{\pi_0(L)} \chi^{-1}Z(L^\nu \cup \gamma)\\
&=& \hbox{``Y'' part of} \ \int_{\pi_0(L)} \chi^{-1} Z(L^\nu\cup \gamma)\\
&=& \hbox{``Y'' part of} \ \int_{\pi_0(L)} \chi^{-1} Z(L\cup \gamma)
\end{eqnarray*}
where the last identity follows from the fact that $\chi^{-1}$  preserves the i-degree filtration.
Next, we have $\chi^{-1} Z(L\cup \gamma) = \exp_\sqcup(A/2) \sqcup \exp_\sqcup(T)$
where the matrix $A:=\Lk_{[-1,1]^3}(L)$ is regarded as a linear combination of struts
and where $T$ consists of connected Jacobi diagrams colored by $\pi_0(L\cup \gamma)$ of i-degree $\geq 1$. 
So we have
$$
\int_{\pi_0(L)} \chi^{-1} Z(L\cup \gamma) = \left\langle \exp_\sqcup(-A^{-1}/2), \exp_\sqcup(T)\right\rangle_{\pi_0(L)}
$$
and we deduce that $z(B,\gamma)$ is the ``Y'' part of $\chi^{-1}Z(L\cup \gamma)$ 
that only involves $\pi_0(\gamma)$, i.e$.$ the ``Y'' part of $\chi^{-1}Z(\gamma)$. 
Besides,  we deduce from Lemma \ref{lem:triple_Milnor} that the Milnor invariant $\bar \mu_{ijk}(\hat \gamma)$ 
for $\hat \gamma \subset \hat B$ is the same as for $\hat \gamma \subset \widehat{[-1,1]^3}=S^3$.
Therefore we can assume in the sequel that $B$ is the standard cube $[-1,1]^3$.

With this assumption on $B$, the invariant $Z(B,\gamma)$ is the usual Kontsevich integral $Z(\gamma)$
of the $q$-tangle $\gamma$ (as it is normalized in \cite{CHM}).
Then we can appeal to \cite{HabeggerMasbaum} or, alternatively, we can proceed as follows. 
If $\gamma$ is the trivial $\ell$-component top tangle,
then the ``Y'' part of  $Z(\gamma)$ is trivial and we are done.
Next, we assume that $\gamma$ is any top tangle in $B=[-1,1]^3$, still with the assumption $\Lk_B(\gamma)=0$. 
Then, according to \cite{Matveev,MN},
$\gamma$ can be obtained from the trivial top tangle 
by a finite number of ``Borromean transformations''. 
Such a transformation performed on the $i$-th, $j$-th and $k$-th components
of $\gamma$ produces in $z(B,\gamma)$ a variation by 
$$
\varepsilon \Ygraphtop{k^+}{j^+}{i^+} 
$$ 
(see \cite[Lemma 11.22]{Ohtsuki}).
Here $\varepsilon=\pm 1$ is a  sign that depends on the configuration of the Borromean rings. 
It is  easily checked that the same  ``Borromean transformation''  
produces for $\bar\mu_{ijk}(\hat{\gamma})$ a variation by $-\varepsilon$. 
\end{proof}

%
%
%
%
%
%

\bibliographystyle{abbrv}

\bibliography{LMO_split}

\end{document}